\documentclass[a4paper]{article}

\usepackage{amsmath,amssymb}
\usepackage{graphicx}
\usepackage{siunitx}
\usepackage{placeins} 

\newcommand{\TheTitle}{On Composite Discontinuous Galerkin Method for simulations of electric properties of semiconductor devices} 

\newcommand{\TheAbstract}{
    In this paper, a variant of discretization of the van Roosbroeck equations in the equilibrium state with the Composite Discontinuous Galerkin Method for the rectangular domain is discussed. It is based on Symmetric Interior Penalty Galerkin (SIPG) method. The proposed method accounts for lower regularity of the solution on the interfaces of devices' layers.
    It is shown that the discrete problem is well-defined and that discrete solution is unique. Error estimates are derived. Finally, numerical simulations are presented.
}

\DeclareMathOperator{\built}{built}
\DeclareMathOperator{\bias}{bias}

\DeclareMathOperator{\lhs}{LHS}
\DeclareMathOperator{\nb}{nb}

\DeclareMathOperator{\rhs}{RHS}

\DeclareMathOperator{\error}{error}

\DeclareMathOperator{\rad}{rad}

\usepackage{amsthm}
\usepackage{hyperref}
\usepackage{authblk}
\DeclareMathOperator{\supp}{supp}

\hypersetup{%
    pdftitle={\TheTitle},
    pdfauthor={Konrad Sakowski, Leszek Marcinkowski, Pawel Strak, Pawel Kempisty, Stanislaw Krukowski},
    pdfkeywords={Composite Discontinuous Galerkin Method, drift-diffusion, van Roosbroeck equations}
    }

\title{\TheTitle}
\date{6. November 2018}

\author[1,3]{Konrad Sakowski\thanks{konrad@unipress.waw.pl}}
\author[3]{Leszek Marcinkowski}
\author[1]{Pawel Strak}
\author[1]{Pawel Kempisty}
\author[2]{Stanislaw Krukowski}
\affil[1]{Institute of High Pressure Physics, Polish Academy of Sciences, ul. Sokolowska 29/37, 01-142 Warsaw, Poland}
\affil[2]{Interdisciplinary Centre for Materials Modeling, Warsaw University, ul. Pawińskiego 5a, 02-106 Warsaw, Poland}
\affil[3]{Department of Mathematics, Computer Science and Mechanics, Warsaw University, ul. Banacha 2, 02-097 Warsaw, Poland}

\newcounter{thmcounter}
\numberwithin{thmcounter}{section}
\newtheorem{theorem}[thmcounter]{Theorem}
\newtheorem{lemma}[thmcounter]{Lemma}
\newtheorem{remark}[thmcounter]{Remark}

\begin{document}

\maketitle

\begin{abstract}
		\TheAbstract
\end{abstract}

\tableofcontents

\section{Introduction}

	Numerical simulations are the important tool in the development of semiconductor devices. Since our contemporary electronics relies on the semiconductors, there is a strong demand for the progress in this domain. Examples of such devices are light emitting diodes, lasers, transistors, detectors, and many others. There are various approaches in simulations of such devices, depending on precision, efficiency, and size of a simulated fragment. On the one hand, there are so-called \emph{ab initio} methods, which are used to investigate properties of elements composed of hundreds of thousands of atoms. These methods use fundamental laws of physics, and they need days or weeks to perform a single simulation on a computational cluster.
	Then there is a drift-diffusion theory. In this case, the model is much simpler, and it allows to simulate whole semiconductor device on a standard desktop computer. This model describes two kinds of carriers (electrons and holes), which move in the electric field present in semiconductor devices. From the mathematical point of view, it consists of a system of three nonlinear elliptic differential equations, which are called the van Roosbroeck equations \cite{BellSystTechJ-1950-29-560}.

	Numerical modelling of semiconductor devices with the drift-diffusion model has been performed since 1964, when Gummel \cite{IEEETransElectDev-1964-11-455} proposed a numerical algorithm based on the simple iteration method. Various methods were used for discretization of the van Roosbroeck equations, for example Finite Difference Method (FDM) \cite{Book-Selberherr-1984}, Box method \cite{SIAMJNumerAnal-1987-24-777}, Finite Element Method (FEM) \cite{Computing-1996-56-1}. Special variants of discretizations optimized for the so-called continuity equations were developed \cite{RepProgPhys-1999-62-277}.

	In this paper, we focus our analysis on the following nonlinear elliptic equation for $u^*$
	\begin{equation}
		\label{eq:poisson:intro}
		\begin{split}
			- \nabla \cdot \Big(\varepsilon \nabla u^* \Big) + e^{u^* - v^*} - e^{w^* - u^*} = k_1,
		\end{split}
	\end{equation}
	which is a special case of the van Roosbroeck problem: find $u^*, v^*, w^*$, such that
	\begin{equation}
	\label{eq:dd}
	\begin{split}
	- \nabla \cdot \left(\varepsilon \nabla u^*\right) + e^{u^* - v^*} - e^{w^* - u^*} = k_1,
	\\
	- \nabla \cdot (\mu_n e^{u^* - v^*} \nabla v^*) - Q(u^*,v^*,w^*) (e^{w^*-v^*} - 1) = 0,
	\\
	- \nabla \cdot (\mu_p e^{w^* - u^*} \nabla w^*) + Q(u^*,v^*,w^*) (e^{w^*-v^*} - 1) = 0.
	\end{split}
	\end{equation}
	Functions $\varepsilon(x), \mu_n(x), \mu_p(x), k_1(x)$ are material parameters and $Q(x,u,v,w)$ is an operator depending on the semiconductor material. We do not want to discuss properties of these equations, we refer to \cite{Book-Selberherr-1984,Book-Markowich-1990} for physical details and to \cite{Book-Jerome-1996} for a mathematical background. In the equilibrium case, when there is no energy exchange between a simulated device and the environment, functions $v^*, w^*$, which correspond to the quasi-Fermi levels \cite{Book-Markowich-1990,JNumerMethEng-1987-24-763}, are constant due to physical nature of this problem and system \eqref{eq:dd} simplifies to \eqref{eq:poisson:intro}.

	We would like to emphasize main problems with the numerical solution of \eqref{eq:dd}.
	The first issue is the nonlinearity. Depending on a device composition and design, the coefficients of the latter two equations may vary by several orders of magnitude. There are various approaches to the solution of this system. They may involve decoupling, Banach iteration \cite{CommPureApplMath-1972-25-781,IEEETransElectDev-1964-11-455}, Newton method \cite{NumerMath-1992-60-525}, etc. In this paper, we do not want to go into detail about this problem. For the specific solution method used by us in numerical simulations, please refer to \cite{LectNotesComputSci-2014-8385-551}.

	The problem we discuss here is the discretization of these equations. As mentioned, FDM and FEM discretizations are successfully used for this system since the second half of 20th century \cite{Book-Selberherr-1984,JNumerMethEng-1987-24-763}. However, a design of the semiconductor devices has been substantially changed over time. Initially, semiconductor transistors or diodes were made from a single material (e.g.,~silicon) divided into layers with different doping level. These conditions were mathematically reflected by variations $k_1$ function, possibly discontinuous, while $\varepsilon, \mu_n, \mu_p$ remained constant.
	On the contrary, contemporary semiconductor devices, like blue laser diodes (see Figure \ref{rys:laser2}), consist of layers of different semiconductor material deposited one on another. Recent designs also involve the change of the material through one layer. The material parameters, like $\varepsilon, \mu_n, \mu_p$, are no longer constant. In general, they are discontinuous. However, these discontinuities are localized on the layers' interfaces, and inside a layer, these parameters are constant or, in general, smooth functions.

	Thus to obtain a good precision, it would be advantageous to use a discretization which takes into account such localized lack of regularity, discontinuities of coefficients on interfaces and which allows exploiting higher regularity inside layers. A natural discretization method for such a problem would be the Discontinuous Galerkin Method (DGM) \cite{Book-Riviere-2008,Book_DiPietro_MathAspDGM}. However, this method by its nature imposes much more degrees of freedom in the simulations, leading to slower and more memory-consuming simulations. Since the physical layers of semiconductor devices have regular shapes, it is feasible to use the Composite Discontinuous Galerkin Method (CDGM) \cite{ComputMethApplMath-2003-3-76}, which is a hybrid between Continuous and Discontinuous Galerkin Method. It allows to divide the domain into subdomains, on which the standard continuous Finite Element Method is used, and on the interfaces between these subdomains, the Interior Penalty method is used, thus allowing for discontinuities. This approach allows to greatly reduce the number of additional degrees of freedom, as they are only needed on the interfaces. Besides, CDGM does not require conforming grids on the interfaces, thus allowing for independent grids for subdomains.

	Composite Discontinuous Galerkin Method is currently successfully developed and used for various problems, for example elliptic eigenvalue problems \cite{ApplMathComput-2015-267-618}, parabolic problems \cite{ComputMethApplMechEng-2001-190-3401}, Darcy flow in homogeneous porous media \cite{IntJSolidStruct-2008-45-6436}. A FETI--DP-type method (Dual Primal Finite Element Tearing and Interconnecting) for CDGM in two dimensions was proposed in \cite{SIAMJNumerAnal-2013-51-400}.

	We aim to use Composite Discontinuous Galerkin Method for semiconductor device simulations due to several reasons.
	First, by its nature, it accounts for separate meshes on the device's layers, so it is possible to use nonconforming grids in general and to tune the mesh for one layer without affecting the rest of the domain.
	Moreover, in simulations of gallium nitride laser diodes, the coefficients of the elliptic equations vary by several orders of magnitude on interfaces between semiconductor material layers. This effect occurs in particular in the active region of semiconductor devices, on interfaces between quantum wells and quantum barriers, on boundaries of the electron blocking layers, etc. Highly varying coefficients are not present in equation \eqref{eq:poisson:intro}, but they occur in the two latter equations of \eqref{eq:dd}. Discontinuous Galerkin Method is more robust than continuous FEM in case of discontinuous, highly variable coefficients. While in this paper we deal with equation \eqref{eq:poisson:intro}, the goal of our study is to use the CDGM method for van Roosbroeck system \eqref{eq:dd}.
	In practice, in physical simulations, we also have to introduce additional physical effects, which are not accounted for by the formulation \eqref{eq:dd}. The important example here is the polarization, which leads to significant interface charges in the nitride-based devices. This effect may be introduced into \eqref{eq:poisson:intro} by addition of the distributional derivatives on the interfaces, which lead to discontinuities of fluxes or unknown functions. In case of Discontinuous Galerkin methods, these discontinuities may be introduced to the model in a very natural way.
	Another reason for using CDGM on the physical background is the local mass conservation, which is a known property of Discontinuous Galerkin Method \cite{Book-Riviere-2008}. This property, in our specific case, corresponds to the Gauss law, while the locality is limited to the subdomains of the device.

	In this paper, we would like to present the error analysis of the CDGM variant for equilibrium state solutions of the van Roosbroeck equations in $\mathbb{R}^2$. We limit our analysis to this case, as the proof framework used in this paper, which is borrowed from the DGM analysis of the Navier-Stokes problem \cite{MathComput-2005-74-53}, imposes the uniqueness of the solution, which is not guaranteed in the non-equilibrium state. For a one-dimensional domain, we have numerical evidence of convergence of the presented method for both equilibrium and non-equilibrium state \cite{LectNotesComputSci-2016-9574-391}. This discretization was also used by our research group in simulations of realistic semiconductor devices \cite{JApplPhys-2012-111-123115}.


	In our analysis, we focus on standard continuous polynomial $\mathbb{P}^k$ element. Simulations, however, are limited to $\mathbb{P}^1$ case only. While there are many computer libraries and frameworks for FEM and DGM discretizations, none that we are aware of supports CDGM out of the box. In particular, it is not possible to define separate meshes across subdomains. Therefore we develop our framework, which currently supports only standard continuous linear $\mathbb{P}^1$ element. While mathematical analysis is presented for equation \eqref{eq:poisson:intro}, in simulations we also cover full drift-diffusion system \eqref{eq:dd}. 

	The remainder of this paper is organized as follows. We start with introduction of the differential problem in Section \ref{sec:differential:problem}. We propose a variant of CDGM discretization of this problem in Section \ref{sec:discretization}. Main result of this paper is stated in Section \ref{sec:main:result}. Then we show existence and uniqueness of the introduced discrete problem in Sections \ref{sec:existence}, \ref{sec:uniqueness}. In Section \ref{sec:interpolation:convergence:2d} we discuss interpolation properties of the discrete space. Then we pass to the error estimate in Section \ref{sec:error:estimates}. Finally we present results of numerical simulations in Section \ref{sec:numerical:experiments} and we conclude in Section \ref{sec:conclusions}.

\section{Differential problem}
	\label{sec:differential:problem}

	The drift-diffusion model describes the relationship between the electrostatic potential and the charge carrier concentrations: electrons and holes \cite{Book-Wilkes-1973,Book-Sze-2006}. The physical derivation of this model is beyond the scope of this work. Therefore we will focus on the mathematical standpoint.

	We start with the domain $\Omega$ of our problem. Luminescent semiconductor devices are made of planar layers deposited one on another, which vary in composition of a semiconductor material or number of impurities (see Figure \ref{rys:laser2}). At opposite ends, metal contacts are attached, where the current can be applied. If this is the case, it flows through the device perpendicular to the deposited layers. We assume that $\Omega$ is a rectangle with boundary $\partial \Omega = \partial \Omega_D \cup \partial \Omega_N$.

	\begin{figure}
		\centering
		\includegraphics[scale=0.65]{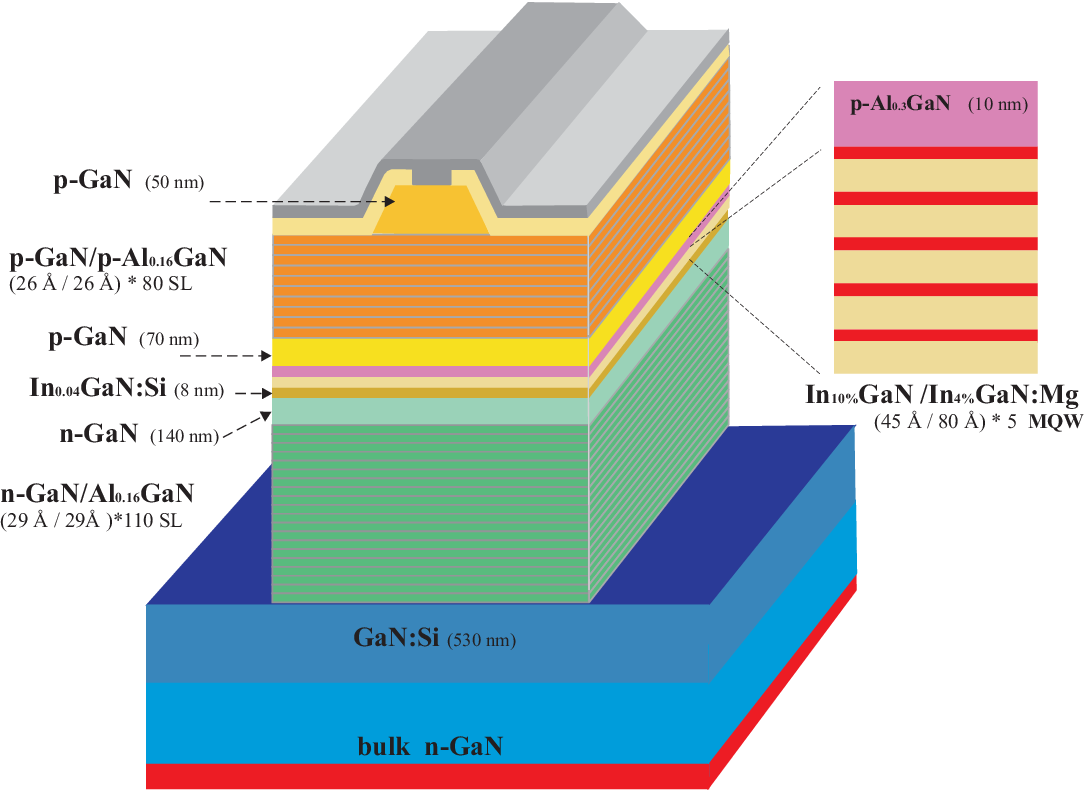}
		\caption{\label{rys:laser2} Example of a gallium nitride semiconductor laser structure.}
	\end{figure}


	In this paper, we deal with equilibrium state. It corresponds to the following differential problem: find $u^* \in H^1(\Omega)$, such that
	\begin{equation}
		\label{eq:probl:diff:u}
		\begin{split}
			- \nabla \cdot \Big(\varepsilon(x) \nabla u^* \Big) + e^{u^* - \hat{v}} - e^{\hat{w} - u^*} = k_1,
			\\
			u^* = \hat{u} \mbox{ on } \partial \Omega_D,
			\\
			\nabla u^* \cdot \nu = 0 \mbox{ on } \partial \Omega_N,
		\end{split}
	\end{equation}
	where $\hat{v} = \hat{w} \equiv \mbox{const}$. Since some results of this paper may be also applied to non-equilibrium case, we consider more general assumption that $\hat{v},\hat{w} \in L_\infty(\Omega)$.
	Also we assume that $\varepsilon, \hat{u} \in H^1(\Omega) \cap L_\infty(\Omega)$ and $0 < \varepsilon_m \leq \varepsilon(x) \leq \varepsilon_M$, $\varepsilon_m, \varepsilon_M \in \mathbb{R}$.

	The following theorem is essential for the results presented in this paper. Its proof may be found in \cite{SIAMJApplMath-1985-45-565}. 
	\begin{theorem}
		Solution $u^*$ of problem \eqref{eq:probl:diff:u} is bounded.
	\end{theorem}

	A weak formulation of the differential problem \eqref{eq:probl:diff:u} is as follows. Find $u^* \in \hat{u} + H^1_0(\Omega)$, such that
	\begin{equation}
		\label{eq:probl:weak:u}
		\begin{split}
			a(u^*,\varphi) + b(u^*,\varphi) = f(\varphi)
			\quad
			\forall \varphi \in H^1_{0,\partial \Omega_D}(\Omega),
		\end{split}
	\end{equation}
	where
	\begin{equation}
		\label{eq:def:a:b:f}
		\begin{split}
			a(u,\varphi)
			:= &
				\int_\Omega \varepsilon(x) \nabla u(x) \cdot \nabla \varphi(x) \, dx,
			\\
			b(u, \varphi)
			:= &
				\int_\Omega  \Big( e^{u(x) - \hat{v}(x)} - e^{\hat{w}(x) - u(x)} \Big) \varphi(x) \, dx,
			\\
			f(\varphi)
			:= &
				\int_\Omega  k_1(x) \varphi(x) \, dx.
		\end{split}
	\end{equation}
	We use the following notation:
	\begin{equation}
		\begin{split}
			C^\infty_{0,\partial \Omega_D} (\overline{\Omega}) := \{ f \in C^\infty(\overline{\Omega}): f|_{\partial \Omega_D} \equiv 0 \},
			\\
			\quad
			H^1_{0,\partial \Omega_D}(\Omega) :=
				\mbox{closure of } C^\infty_{0,\partial \Omega_D} (\overline{\Omega})
				\mbox{ in }H^1(\Omega).
		\end{split}
	\end{equation}

\section{Discretization}
	\label{sec:discretization}

	\subsection{Discrete space}
		\label{sec:discrete:space:2d}
	
		\begin{figure}
			\centering
			\includegraphics[scale=1]{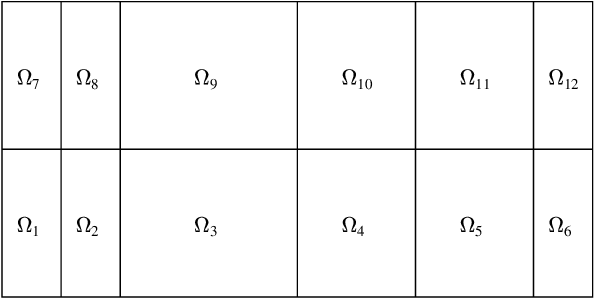}
			\caption{\label{rys:omega2d} An example of two-dimensional coarse grid of $\Omega$.}
		\end{figure}

		Let $\Omega \subset \mathbb{R}^2$ be a rectangle, divided to disjoint subrectangles $\{\Omega_i\}_{i=1}^{N} =: \mathcal{E}$ in such a manner that $\mathcal{E}$ is a conforming mesh \cite{Book_Toselli_DomaDeco} of $\Omega$ (Figure \ref{rys:omega2d}).
		We will call this division a coarse grid, and we assume that if $e \subset \partial \Omega$ is an edge of some $\Omega_i$, then either $e \subset \partial \Omega_D$ or $e \subset \partial \Omega_N$.

		Let us define triangulations $\mathcal{T}_{h_i} := \mathcal{T}_{i,h_i}(\Omega_i)$, where $h_i := \max\{\mathrm{diam}(\tau): \tau \in \mathcal{T}_{h_i}\}$. By $\mathcal{N}_{h_i}$ we denote the nodes of the triangulation $\mathcal{T}_{h_i}$.
		We assume that $\{\mathcal{T}_{i,h_i}(\Omega)\}_{h_i}$ is a regular uniform family of triangulations \cite{Book-Brenner-2008}.
		We will define
		$
			\mathcal{T}_h := \bigcup_{i=1}^{N} \mathcal{T}_{h_i}.
		$
		For $s > 0$, we define the broken Sobolev spaces $H^s(\mathcal{E})$ and $H^s(\mathcal{T}_h)$ as
		\begin{equation}
			\label{eq:def:broken:sob}
			\begin{split}
				H^s(\mathcal{E}) & := \{v \in L_2(\Omega): \forall i \in \{1,\ldots,N\} \; \; v_i := v|_{\Omega_i} \in H^s(\Omega_i)\}
				\subset L_2(\Omega),
				\\
				H^s(\mathcal{T}_h) & := \{v \in L_2(\Omega): \forall \tau \in \mathcal{T}_h \; \; v|_{\tau} \in H^s(\tau)\}
				\subset L_2(\Omega).
			\end{split}
		\end{equation}
		Then on every $\Omega_i$ we define a discrete space $X_{h_i}(\Omega_i)$ of piecewise polynomial functions on the triangulation $\mathcal{T}_{h_i}$:
		\begin{equation}
			\label{eq:def:x:hi:2d}
			X_{h_i} := X_{h_i}(\Omega_i) :=
			\Big\{
				u_{h,i} \in \mathcal{C}(\overline{\Omega}_i): 
				\forall \tau \in \mathcal{T}_{h_i} 
				\, \, \,
				u_{h,i}\big|_\tau \in \mathbb{P}^k(\tau)
			\Big\},
		\end{equation}
		where $k \geq 1$ is some integer.
		Finally we define $X_h(\Omega)$ as
		\begin{equation}
			\label{eq:def:x:h:2d}
			\begin{split}
				X_h(\Omega)
					& =
					X_{h_1}(\Omega_1) \times \cdots \times X_{h_N}(\Omega_N).
			\end{split}
		\end{equation}
		Note that we may treat any element of $X_h(\Omega)$ as a piecewise-continuous function, which values are determined up to interfaces $\partial \Omega_i \cap \partial \Omega_j$. Thus we identify $X_h(\Omega)$ with a suitable subset of $L_2(\Omega)$ space.
		Then note that $X_h(\Omega) \not \subset H^1(\Omega)$ and $X_h(\Omega) \not \subset H^2(\mathcal{E})$, but $X_h(\Omega) \subset H^1(\mathcal{E})$, $H^1(\Omega) \subset H^1(\mathcal{E})$ and $X_h(\Omega) \subset H^2(\mathcal{T}_h)$. 

		By $\Gamma$ we denote a set of all internal and boundary edges of $\mathcal{E}$. Then $\Gamma$ is a sum of disjoint sets $\Gamma_D$, $\Gamma_N$ and $\Gamma_I$, where
		\begin{equation}
			\begin{split}
				\Gamma_D & := \{ e \in \Gamma: e \subset \partial \Omega_D\},
				\Gamma_N := \{ e \in \Gamma: e \subset \partial \Omega_N\},
				\Gamma_I := \{ e \in \Gamma: e \subset \mathrm{int}(\Omega)\}.
			\end{split}
		\end{equation}
		Therefore $\Gamma_D$ (resp. $\Gamma_N$) contains edges lying on the boundary where Dirichlet (resp. Neumann) boundary conditions are imposed and in $\Gamma_I$ there are all internal edges, which we call interfaces, as they frequently correspond to the physical interfaces between different semiconductor materials.
		We also define
		\begin{equation}
			\begin{split}
				\Gamma_{DI} & := \Gamma_D \cup \Gamma_I,
				\quad
				\Gamma_i := \{e \in \Gamma: e \subset \partial \Omega_i \}.
			\end{split}
		\end{equation}
		Let $e \in \Gamma$. Then two cases are possible. Either $e \in \Gamma_D \cup \Gamma_N$, so there is an unique $\Omega_i$ such that $e$ is an edge of $\Omega_i$, or $e \in \Gamma_I$ and there are exactly two sets $\Omega_i, \Omega_j \in \mathcal{E}$ such that $e$ is their common edge. Also we define $\nb(\Omega_i) := \{\Omega_l \in \mathcal{E}: \Gamma_i \cap \Gamma_l \neq \emptyset\}$.
		Moreover, for $e \in \Gamma_D \cup \Gamma_N$ by $\nu$ we denote the normal vector to $\Omega$. On the other hand, for $e \in \Gamma_I, e = \partial \Omega_i \cap \partial \Omega_j, i<j$ we define $\nu$ to be a vector normal to $\Omega_i$.
		Thus also $-\nu$ is normal to $\Omega_j$. Opposite direction of these vectors may also be used, but they must be used consequently.
		
		For $s > 1/2$ we define operators $[\cdot] := [\cdot]_e : H^s(\mathcal{E}) \rightarrow L_2(e)$, $\{\cdot\} := \{\cdot\}_e : H^s(\mathcal{E}) \rightarrow L_2(e)$ as
		\begin{equation}
			\begin{split}
				[u] & :=
				\begin{cases}
					u_{i} - u_{j} &\text{ if } e \subset \Gamma_I, e = \partial \Omega_i \cap \partial \Omega_j, i<j\\
					u_{i} &\text{ if } e \subset \Gamma_D \cup \Gamma_N, e = \partial \Omega_i \cap \partial \Omega,\\
				\end{cases}
				\\
				\{u\} & :=
				\begin{cases}
					\frac{1}{2} \big( u_{i} + u_{j} \big) &\text{ if } e \subset \Gamma_I, e = \partial \Omega_i \cap \partial \Omega_j,\\
					u_{i} &\text{ if } e \subset \Gamma_D \cup \Gamma_N, e = \partial \Omega_i \cap \partial \Omega.\\
				\end{cases}
			\end{split}
		\end{equation}
		For convenience, we will also use this notion for triangulation parameters, i.e.,
		\begin{equation}
			\begin{split}
				 \{h^{-s}\} := \Big\{\frac{1}{h^s}\Big\} & :=
				\begin{cases}
					\frac{1}{2} \Big( \frac{1}{h^s_{i}} + \frac{1}{h^s_{j}} \Big) &\text{ if } e = \partial \Omega_i \cap \partial \Omega_j, \\
					\frac{1}{h^s_{i}} &\text{ if } e = \partial \Omega_i \cap \partial \Omega.
				\end{cases}
			\end{split}
		\end{equation}
		
		For further analysis, we introduce so-called broken norm $\|\cdot\|_{h}$ in $X_h(\Omega)$ as
		\begin{equation}
			\label{def:broken:norm:2d}
			\|u_{h}\|_{h}^2
			:=
				\sum_{i=1}^N \int_{\Omega_i} \varepsilon \Big( \nabla u_{h,i} \Big)^2\, dx
				+
				\sum_{e \in \Gamma_{DI}} \eta_{e} \int_e [ u_h ]^2\, ds
				.
		\end{equation}
		where
		\begin{equation}
			\label{eq:def:eta:e}
			\eta_{e} := 2 \sigma_e \{h^{-1}\}
				=
				\begin{cases}
					2 \sigma_e  h_i^{-1}   & e \in \Gamma_D, e \subset \Omega_i,\\
					\sigma_e \Big( h_i^{-1} + h_j^{-1} \Big)  & e \in \Gamma_I, e \subset \Omega_i \cap \Omega_j.
				\end{cases}
		\end{equation}
		Here $\sigma_e > 0$ is a penalty parameter.

		To simplify the analysis, we assume that $0<\sigma_0 \leq \sigma_e$ for all $e \in \Gamma_{DI}$. Also we assume that $0<h_i < h_0 \leq 1$ for all $i \in \{1, \ldots, N\}$. The choice of $\sigma_0$ and $h_0$ will be discussed later in lemmas \ref{lem:broken:koerc:symm} and \ref{lemma:poincare:broken:norm:2d}.

		We also need the following standard result for FEM spaces:
		\begin{lemma}
			\label{lemma:interface:trace:estimates}
			For any $u_h \in X_h(\Omega)$, $\Omega_i \in \mathcal{E}$ and $e \in \Gamma_i$, the following estimates hold
			\begin{eqnarray}
				\|u_{h,i} \|_{L_2(e)}
				& \leq &
					C h_i^{-1/2} \| u_{h,i} \|_{L_2(\Omega_i)}
				, \\
				\|\nabla u_{h,i} \cdot \nu \|_{L_2(e)}
				& \leq &
					C h_i^{-1/2} | u_{h,i} |_{H^1(\Omega_i)}
				.
			\end{eqnarray}
			Constant $C$ does not depend on $h_i$.
		\end{lemma}
		These estimates are a consequence of the trace theorem applied to each edge of fine elements in $\Omega_i$ coincident with $e$ followed by a scaling argument.

	\subsection{Discrete problem}

		We discuss a variant of the Composite Discontinuous Galerkin discretization, derived from Symmetric Interior Penalty Galerkin (SIPG) method (cf. \cite{Book-Riviere-2008} or \cite{Book_DiPietro_MathAspDGM}). We use the composite formulation (cf. \cite{ComputMethApplMath-2003-3-76}), i.e., inside every $\Omega_i$ we use the Finite Element Method on the triangulation $\mathcal{T}_{h_i}$, while on boundaries $e \in \Gamma_{DI}$ we use the Discontinuous Galerkin Method.

		This problem is defined as follows. Find $u_h^* \in X_h(\Omega)$ such that 
		\begin{equation}
			\label{eq:std:symm:problem}
			a_{h}(u_h^*,\varphi_h)
			+ b(u_h^*,\varphi_h)
			=
				f_{h}(u_h^*,\varphi_h),
			\quad
			\forall \varphi_h \in X_h(\Omega),
		\end{equation}
		where
		\begin{equation}
			\label{eq:ah:b:fh:symm}
			\begin{split}
				a_{h}(u_h,\varphi_h)
					 = &
						\sum_{i=1}^N \int_{\Omega_i}  \varepsilon \nabla u_{h,i} \cdot \nabla \varphi_{h,i}\, dx
					-
					\sum_{e \in \Gamma_{DI}} \int_e \{\varepsilon \nabla u_h \cdot \nu\} [\varphi_h]\, ds
				\\ &
					-
					\sum_{e \in \Gamma_{DI}} \int_e \{\varepsilon \nabla \varphi_h \cdot \nu\} [u_h]\, ds
					+
					\sum_{e \in \Gamma_{DI}} \eta_{e} \int_e [ u_h ] \cdot [ \varphi_h ]\, ds
					,
				\\
				f_{h}(\varphi_h)
				 = &
					\int_{\Omega} k_1 \varphi_h\, dx
					-
					\sum_{e \in \Gamma_{D}} \int_e \{\varepsilon \nabla \varphi_h \cdot \nu\} [\hat{u}]\, ds
				\\ &
					+
					\sum_{e \in \Gamma_D} \eta_{e} \int_e [ \hat{u} ] [ \varphi_h ]\, ds
				,
			\end{split}
		\end{equation}
		and $b(u,\varphi)$ is defined as in \eqref{eq:def:a:b:f}.

\section{Main result}
	\label{sec:main:result}

	Most of this paper is dedicated to justifying the following result.
	\begin{theorem}
		\label{thm:main:estimate}
		\
		\begin{enumerate}
			\item[(a)] The solution $u_h^* \in X_h(\Omega)$ of discrete problem \eqref{eq:std:symm:problem} exists and it is unique.
			\item[(b)] Assume that $u^* \in H^1(\Omega) \cap H^{k+1}(\mathcal{E})$, $k \geq 1$, is a solution of differential problem \eqref{eq:probl:weak:u} and $\varepsilon \in L_\infty(\Omega)$, $\varepsilon|_{\Omega_i} \in \mathcal{C}^1(\overline{\Omega_i})$ for all $i \in \{1, \ldots, n\}$.
			Then the following error estimate holds:
			\begin{equation}
				\label{eq:main:estimate}
				\begin{split}
					\|u^*-u_h^*\|_{h}
					\leq &
						\|u^*-u^*_I\|_{h}+\|u^*_I-u_h^*\|_{h}
					\\
					\leq &
						C \Bigg(
							\sum_{i=1}^N \Big( h_i^{2k} + \sum_{\Omega_l \in \nb(\Omega_i)} \frac{h_i^{2k+1}}{h_l}
							\Big) |u^*_i|_{H^{k+1}(\Omega_i)}^2
						\Bigg)^{1/2}
					.
				\end{split}
			\end{equation}
		\end{enumerate}
	\end{theorem}
	\begin{remark}
		\label{remark:main:estimate}
		If additionally we assume that $h_i := c_i h$ for every $\Omega_i \in \mathcal{E}$, then estimate \eqref{eq:main:estimate} reduces to
		\begin{equation}
			\begin{split}
				\|u^*-u_h^*\|_{h}
				\leq &
					C h^k \Big(
						\sum_{i=1}^N |u^*_i|_{H^{k+1}(\Omega_i)}^2 \Big)^{1/2}
				.
			\end{split}
		\end{equation}
	\end{remark}
	Existence and uniqueness of the discrete problem are shown in Sections \ref{sec:existence} and \ref{sec:uniqueness}, respectively. Then the error estimate is derived in Section \ref{sec:error:estimates}.

\section{Existence}
	\label{sec:existence}

	We define $P:X_h(\Omega) \rightarrow X_h^*(\Omega)$ as
	\begin{equation}
		\label{eq:def:P:oper}
		P(u_h) \varphi_h :=
			a_{h}\big(u_h, \varphi_h\big)
			+ b \big(u_h, \varphi_h\big)
			- f_{h}\big(\varphi_h\big)
			.
	\end{equation}
	We would like to use the following consequence of the Brouwer theorem \cite{Book_Girault_FiniElem,Book_Lions_QuelMeth}:
	\begin{theorem}
		\label{thm:brouwer:variant}
		Let $P:X \rightarrow X^*$ be a continuous function on a finite-dimensional normed real vector space $X$, such that for suitable $\rho>0$ we have
		\begin{equation}
			P(x) x \geq 0
			\quad
			\forall \|x\| \geq \rho.
		\end{equation}
		Then there exists $x \in X$ such that
		\begin{equation}
			P(x) = 0.
		\end{equation}
	\end{theorem}

	We need the following lemma, which is a simple consequence of the Schwarz inequality and the Cauchy's $\epsilon$-inequality (see also \cite{ComputMethApplMath-2003-3-76}):
	\begin{lemma}
		\label{lem:broken:koerc:symm}
		There exist $\sigma_0 > 0$ and $c > 0$, such that for every  $\sigma_e \geq \sigma_0$ and $u_h \in X_h(\Omega)$
		\begin{equation}
			c \|u_h\|_{h}^2 \leq a_{h}(u_h, u_h).
		\end{equation}
		Constant $\sigma_0$ depends on $\varepsilon_m, \varepsilon_M$ and the maximal number of edges of elements in coarse grid $\mathcal{E}$.
	\end{lemma}

	Let $C := \max \{\|\hat{v}\|_{L_\infty(\Omega)}, \|\hat{w}\|_{L_\infty(\Omega)}\}$.
	Then we may decompose $b(u_h, u_h)$ as
	\begin{equation}
		\begin{split}
			b\big(u_h, u_h\big)
			= &
				\int_{\Omega}
					\Big(
						e^{u_h - \hat{v}}
						-
						e^{\hat{w} - u_h}
					\Big)
					u_h
				dx
			\\
			= &
				\int_{\Omega}
					\Big(
						e^{u_h - \hat{v}}
						-
						e^{\hat{w} - u_h}
					\Big)
					u_h
					\chi_{\{x \in \Omega : |u_h (x)| > C\}}
				dx
			\\ &
				+
				\int_{\Omega}
					\Big(
						e^{u_h - \hat{v}}
						-
						e^{\hat{w} - u_h}
					\Big)
					u_h
					\chi_{\{x \in \Omega : |u_h (x)| \leq C\}}
				dx
			.
		\end{split}
	\end{equation}
	The first integral is non-negative, and the latter we can estimate from below
	\begin{equation}
		\label{eq:exist:c3}
		\int_{\Omega}
			\Big(
				e^{u_h(x) - \hat{v}(x)}
				-
				e^{\hat{w}(x) - u_h(x)}
			\Big)
			u_h(x)
			\chi_{\{x \in \Omega : |u_h(x)| \leq C\}}(x)
		dx
		\geq
			- |\Omega|
			2 e^{2 C}
			C
		.
	\end{equation}
	To estimate $f_{h}(u_h)$ we first use lemma \ref{lemma:interface:trace:estimates} and the trace inequality
	\begin{equation}
		\begin{split}
			\Big|
				\sum_{e \in \Gamma_{D}} \int_e \{\varepsilon \nabla u_h 
				 \cdot \nu\} [\hat{u}] \, ds
			\Big|
			& \leq
				\varepsilon_M \sum_{e \in \Gamma_{D}} \| \nabla u_h\|_{L_2(e)} \|\hat{u}\|_{L_2(e)}
			\\ &
			\leq
				c \varepsilon_M \sum_{i=1}^N h_i^{-1/2} \| \nabla u_h\|_{L_2(\Omega_i)} \|\hat{u}\|_{L_2(\Omega_i)}
			\\ &
			\leq
				C \| u_h\|_{h} \|\hat{u}\|_{H_1(\Omega)}
			,
		\end{split}
	\end{equation}
	where $C$ depends on $\varepsilon_M$ and $h$.
	Then using the Schwarz inequality, we obtain
	\begin{equation}
		\label{eq:exist:c2}
		- f_{h}\big(u_h\big)
		\geq - c (\hat{u}, k_1, h) \|u_h\|_{h}
		.
	\end{equation}
	Therefore by lemma \ref{lem:broken:koerc:symm} and by \eqref{eq:exist:c3}, \eqref{eq:exist:c2} we get
	\begin{equation}
		\begin{split}
			P(u_h) u_h
			\geq
				c_1 \|u_h\|_{h}^2
				- c_2 \|u_h\|_{h}
				- c_3
			,
		\end{split}
	\end{equation}
	where $c_1, c_2, c_3 \in \mathbb{R}$ are some positive constants independent of $u_h$.
	It is therefore clear that for $\|u_h\|_{h}$ large enough, we have that $P(u_h) u_h \geq 0$.
	Then by theorem \ref{thm:brouwer:variant} there exists some $u_h^* $, such that $P(u_h^*) = 0$.

\section{Uniqueness}
	\label{sec:uniqueness}

	Assume that there exist two solutions $u_h^*, v_h^* \in X_h(\Omega)$ of equation \eqref{eq:std:symm:problem}.
	Then taking $\varphi_h := u_h^* - v_h^*$ and subtracting \eqref{eq:std:symm:problem} for $u_h^*$ and $v_h^*$ we obtain
	\begin{equation}
		\begin{split}
			a_{h}(u_h^* - v_h^*, u_h^* - v_h^*)
			& =
				\sum_{i=1}^N \int_{\Omega_i} e^{-\hat{v}} \Big(
					e^{v_h^*}
					- e^{u_h^*}
				\Big)
				\Big( u_h^* - v_h^* \Big)\, dx
			\\ &
				+
				\sum_{i=1}^N \int_{\Omega_i} e^{\hat{w}} \Big(
					e^{ - u_h^*}
					- e^{ - v_h^*} 
				\Big)
				\Big( u_h^* - v_h^* \Big)\, dx
				.
		\end{split}
	\end{equation}
	By the monotonicity of the exponential function, the right hand side is nonpositive.
	On the other hand by lemma \ref{lem:broken:koerc:symm} we have
	\begin{equation}
			0
			 <
				c \|u_h^* - v_h^*\|_{h}^2
			 \leq 
				 a_{h}(u_h^* - v_h^*, u_h^* - v_h^*).
	\end{equation}
	Thus
	$
			0 < \|u_h^* - v_h^*\|_{h}^2 \leq 0
	$
	since $u_h^* \neq v_h^*$, and we have a contradiction.

\section{Interpolation operator}
	\label{sec:interpolation:convergence:2d}

	For any $\Omega_i \in \mathcal{E}$ let $\Pi_{h_i} : H^{k+1}(\Omega_i) \rightarrow X_{h_i} \subset \mathcal{C}^0(\overline{\Omega_i})$ be a standard piecewise-polynomial continuous interpolation operator.
	Then we define $\Pi_h: H^{k+1}(\mathcal{E}) \rightarrow X_h$ by
	\begin{equation}
		\forall \Omega_i \in \mathcal{E}
		\quad
		\Big( \Pi_h u \Big)_i := \Pi_{h_i} u_i.
	\end{equation}
	On any $\Omega_i$, we can use standard interpolation estimate for FEM \cite{Book-Brenner-2008}:
	\begin{equation}
		\begin{split}
			\label{eq:std:inter:res}
			\|u_i - \Pi_{h_i} u_i \|_{L_2(\Omega_i)}
			+
			h_i \|u_i - \Pi_{h_i} u_i \|_{H^1(\Omega_i)}
			& \leq
				C h_i^{k+1} |u_i|_{H^{k+1}(\Omega_i)}.
		\end{split}
	\end{equation}
	Let further $u_I := \Pi_h u$.

	\begin{lemma}
		\label{lemma:estimate:inter:e}
		Let $u \in H^{k+1}(\mathcal{E})$, $u_I := \Pi_h u$. For any $\Omega_i \in \mathcal{E}$ and for any $e \in \Gamma_i$
		\begin{eqnarray}
			\big\| u_i - u_{I,i} \big\|_{L_2(e)}
			& \leq &
				C h_i^{k+1/2}
				| u_i |_{H^{k+1}(\Omega_i)}
			,
			\\
			\big| u_i - u_{I,i} \big|_{H^1(e)}
			& \leq &
				C h_i^{k-1/2}
				| u_i |_{H^{k+1}(\Omega_i)}
			.
		\end{eqnarray}
	\end{lemma}

	\begin{proof}
		For fixed $e$ and $\Omega_i$, we have
		$
				\|u_i - u_{I,i} \|_{L_2(e)}^2
				 =
					\sum_{\tau \in \mathcal{T}_{h_i,e}} 
						\|u_i - u_{I,i} \|_{L_2(e \cap \tau)}^2
		$.
		Note that on a single triangulation element $\tau$ we have that $u_i - u_{I,i} \in H^2(\tau)$, so using the trace inequality (see \cite{Book-Riviere-2008}) for $H^2(\tau)$ functions we have
		\begin{equation}
			\begin{split}
				\|u_i - u_{I,i} \|_{L_2(e \cap \tau)}
				& = 
					C \Big(
						 h_i^{-1/2} \|u_i - u_{I,i} \|_{L_2(\tau)}
						+
						h_i^{1/2} \big|u_i - u_{I,i} \big|_{H_1(\tau)}
						\Big)
						.
			\end{split}
		\end{equation}
		Then by \eqref{eq:std:inter:res} it follows that
		\begin{equation}
			\label{eq:l2:interp:est}
			\begin{split}
				\|u_i - u_{I,i} \|_{L_2(e)}
				& \leq
					C h_i^{-1/2} \Big(
						\|u_i - u_{I,i} \|_{L_2(\Omega_i)}
						+
						h_i
							\big\|u_i - u_{I,i}\big\|_{H^1(\Omega_i)}
						\Big)
				\\ & \leq
					C h_i^{k+1/2} |u_i|_{H^{k+1}(\Omega_i)}
				.
			\end{split}
		\end{equation}
		Proof of the latter estimate is analogous.
	\end{proof}

	Let us take any $e \in \Gamma_{DI}$. For $e \in \Gamma_I$ we assume that $e = \Omega_j \cap \Omega_l$ for some $\Omega_j, \Omega_l \in \mathcal{E}$ and by the triangle inequality we have
	\begin{equation}
		\begin{split}
			\int_e [ u - u_I ]^2\, ds
			& =
				\| [ u - u_I ] \|_{L_2(e)}^2
			\leq
				2 \sum_{i \in \{j,l\}} \|u_i - u_{I,i} \|_{L_2(e)}^2
				,
		\end{split}
	\end{equation}
	while for $e \in \Gamma_D$ we have $e \in \Gamma_i$ for some $\Omega_i \in \mathcal{E}$ and simply
	\begin{equation}
		\int_e [ u - u_I ]^2\, ds
		= 
			\int_e \Big( u_i - u_{I,i} \Big)^2\, ds
		= 
			\|u_i - u_{I,i} \|_{L_2(e)}^2
		.
	\end{equation}
	Therefore it is sufficient to estimate
	$
		\|u_i - u_{I,i} \|_{L_2(e)}^2,
	$
	for any $e \in \Gamma_{DI}$, $e \subset \partial \Omega_i$, using lemma \ref{lemma:estimate:inter:e}.
	Let $e \in \Gamma_D$.
	By \eqref{eq:l2:interp:est} we have
	\begin{equation}
		\begin{split}
			\eta_{e} \int_e \Big( u_i - u_{I,i} \Big)^2\, ds
				& = 
					\sigma_e h_i^{-1} \|u_i - u_{I,i} \|_{L_2(e)}^2
				\\ &
				\leq
					C \sigma_e h_i^{-1} h_i^{2k+1} |u_i|_{H^{k+1}(\Omega_i)}^2
				=
					C \sigma_e h_i^{2k} |u_i|_{H^{k+1}(\Omega_i)}^2
					.
		\end{split}
	\end{equation}
	On the other hand, if $e \in \Gamma_I$ then
	\begin{equation}
		\begin{split}
			\eta_{e} \int_e \Big( u_j - u_{I,j} \Big)^2\, ds
				&
				=
					0.5 \sigma_e (h_j^{-1} + h_l^{-1}) \|u_j - u_{I,j} \|_{L_2(e)}^2
				\\ &
				\leq
					C \sigma_e \Big(h_j^{2k} + \frac{h_j^{2k+1}}{h_l}\Big) |u_j|_{H^{k+1}(\Omega_j)}^2
					.
		\end{split}
	\end{equation}
	Then if we sum up over $e \in \Gamma_{DI}$ 
	\begin{equation}
		\sum_{e \in \Gamma_{DI}} \eta_{e} \int_e \Big( [ u - u_I ] \Big)^2\, ds
		\leq
			\sum_{i=1}^N
				C \Bigg(h_i^{2k} + \sum_{\Omega_l \in \nb(\Omega_i)} \frac{h_i^{2k+1}}{h_l}\Bigg) |u_i|_{H^{k+1}(\Omega_i)}^2
				.
	\end{equation}
	Thus taking into account this estimate and \eqref{eq:std:inter:res}
	\begin{equation}
		\label{eq:inter:est:2d:csipg}
		\begin{split}
			\|u - u_I\|_{h}^2
			\leq
				C \sum_{i=1}^N
					\Bigg(h_i^{2k} + \sum_{\Omega_l \in \nb(\Omega_i)} \frac{h_i^{2k+1}}{h_l}\Bigg) |u_i|_{H^{k+1}(\Omega_i)}^2.
		\end{split}
	\end{equation}
	If we increase density proportionally, i.e., $h_i := c_i h$, the result can be improved to
	\begin{equation}
		\begin{split}
			\|u - u_I\|_{h}^2
			\leq
				C h^{2k} \sum_{i=1}^N |u_i|_{H^{k+1}(\Omega_i)}^2.
		\end{split}
	\end{equation}

\section{Error estimates}
	\label{sec:error:estimates}

	We start with the following auxiliary lemma.
	\begin{lemma}
		\label{lemma:poincare:auxiliary}
		Let $u \in H^s(\mathcal{E})$, $s \geq 1$. Then
		\begin{equation}
			\|u\|_{L_2(\Omega)}^2
			\leq
				C \Big[
					\sum_{i=1}^N \int_{\Omega_i} \big( \nabla u \big)^2\, dx
					+
					\sum_{e \in \Gamma_I} |e|^{-1} \int_e [u]^2\, ds
					+
					\sum_{e \in \Gamma_D} \int_{e} u^2\, ds
				\Big]
				.
		\end{equation}
	\end{lemma}

	Proof of lemma \ref{lemma:poincare:auxiliary} may be found in \cite{SIAMJNumerAnal-2004-41-306}.

	Then we would like to have an analog of a Poincare inequality for the $H^s(\mathcal{E})$ spaces.
	\begin{lemma}
		\label{lemma:poincare:broken:norm:2d}
		Let $u \in H^s(\mathcal{E})$, $s \geq 1$. Then there exists some $h_0>0$, such that 
		$\|u\|_{L_2(\Omega)} \leq c \|u\|_{h}$ for $0<h \leq h_0$, where $c$ is independent of $h$.
	\end{lemma}

	\begin{proof}
		By definition of the broken norm \eqref{def:broken:norm:2d}, we have
		\begin{equation}
			\|u\|_{h}^2
			:=
				\sum_{i=1}^N \int_{\Omega_i} \varepsilon_i \Big( \nabla u \Big)^2\, dx
				+
				\sum_{e \in \Gamma_{DI}} \eta_{e} \int_e [ u ]^2\, ds
			.
		\end{equation}
		Note that $|e|$ does not depend on $h$ and $\eta_{e} \rightarrow \infty$ as $h \rightarrow 0$. Thus we can find $h_0>0$, such that $\eta_{e} \geq |e|^{-1}$ and $\eta_{e} \geq 1$ for any $0<h<h_0$ and then by lemma \ref{lemma:poincare:auxiliary}
		\begin{equation}
			\begin{split}
				\|u\|_{L_2(\Omega)}^2
				& \leq
					C \Big[
						\sum_{i=1}^N \int_{\Omega_i} \big( \nabla u \big)^2\, dx
						+
						\sum_{e \in \Gamma_I} |e|^{-1} \|[u]\|_{L_2(e)}^2
						+
						\sum_{e \in \Gamma_D} \|[u]\|_{L_2(e)}^2
					\Big]
				\\
				& \leq
					C \Big[
						\varepsilon_m^{-1} \sum_{i=1}^N \int_{\Omega_i} \varepsilon_i \big( \nabla u \big)^2\, dx
						+
						\sum_{e \in \Gamma_{DI}} \eta_{e} \int_e [u]^2\, ds
					\Big]
				\\
				& \leq
					C_1 \|u\|_{h}^2
				.
			\end{split}
		\end{equation}
	\end{proof}

	To prove error estimates of the proposed discretization, we would like to introduce the following assumptions:
	\begin{equation}
		u^*\in H^{1}(\Omega)\cap H^{k+1}(\mathcal{E}),
		\quad
		\varepsilon\in \left\{v \in L_\infty(\Omega): \forall i \in \{1,\ldots,n \} \, v|_{\Omega_i} \in \mathcal{C}^1(\overline{\Omega}) \right\}.
	\end{equation}

	\subsection{Consistency}

		We start with an abstract result. Let $f \in L_2(\Omega)$. We pose two problems. The first is the following: find $u^* \in H^1(\Omega)$ such that
		\begin{equation}
			\label{eq:con:probl1:gen}
			\begin{split}
				\int_\Omega \varepsilon \nabla u^* \cdot \nabla \varphi\, dx & = \int_\Omega f \varphi\, dx  \quad \forall \varphi \in  H^1_{0,\partial \Omega_D}(\Omega),\\
				u^* & =\hat{u}  \quad \mbox{on } \partial \Omega_D.
			\end{split}
		\end{equation}
		Second problem is posed in broken Sobolev space: find $u^* \in H^1(\mathcal{E})$, such that $\forall \varphi \in H^1(\mathcal{E}) \cap H^2(\mathcal{T}_h)$
		\begin{equation}
			\label{eq:con:probl2:gen}
			\begin{split}
				\sum_{i=1}^N \int_{\Omega_i} & \varepsilon \nabla u^* \cdot \nabla \varphi\, dx
				 -
				\sum_{e \in \Gamma_{DI}} \int_e \Big\{ \varepsilon \nabla u^* \cdot \nu \Big\} [\varphi]\, ds
			\\ & 
				-
				\sum_{e \in \Gamma_{DI}} \int_e \Big\{ \varepsilon \nabla \varphi \cdot \nu \Big\} [u^*]\, ds
				 +
				\sum_{e \in \Gamma_{DI}} \eta_{e} \int_e[u^*] [\varphi]\, ds
			 \\ &
			 =
				\sum_{i=1}^N \int_{\Omega_i} f \varphi\, dx
				-
				\sum_{e \in \Gamma_{D}} \int_e \Big\{ \varepsilon \nabla \varphi \cdot \nu \Big\} [\hat{u}]\, ds
				+
				\sum_{e \in \Gamma_D} \eta_{e} \int_e [\hat{u}] [\varphi]\, ds
				.
			\end{split}
		\end{equation}
		We would like to prove the following result.
		\begin{theorem}
			\label{thm:u:consistency:gen}
			Assume that the solution $u^*$ of problem \eqref{eq:con:probl1:gen} belongs to $H^1(\Omega) \cap H^2(\mathcal{E})$ and $\varepsilon \nabla u^* \in H^1(\mathcal{E})$. Then $u^*$ satisfies \eqref{eq:con:probl2:gen}.
			Conversely, if $u^* \in H^2(\mathcal{E}) \cap H^1(\Omega)$ is a solution of \eqref{eq:con:probl2:gen} and $\varepsilon \nabla u^* \in H^1(\mathcal{E})$, then it is also a solution of \eqref{eq:con:probl1:gen}. 
		\end{theorem}

		The proof presented in this paper is based on the standard approach in Discontinuous Galerkin Method, cf. e.g., \cite{Book-Riviere-2008}.



		\begin{lemma}
			\label{lem:elipt:rownowazn:gen}
			Let $u \in H^1(\Omega) \cap H^2(\mathcal{E})$, $\varepsilon \nabla u \in \big(H^1(\mathcal{E})\big)^2$, $0<\varepsilon_m\leq \varepsilon \leq \varepsilon_M$ and $f \in L_2(\Omega)$.
			The following statements are equivalent:
			\begin{itemize}
				\item 
					$u$ satisfy:
					\begin{equation}
						\label{eq:elipt:rownowazn:gen:1}
						\begin{split}
							\int_\Omega \varepsilon \nabla u \cdot \nabla \varphi &= \int_\Omega f \varphi,
							\quad
							\forall \varphi \in H^1_{0,\partial \Omega_D}(\Omega)
							,
						\end{split}
					\end{equation}
				\item
					$u$ satisfy:
					\begin{equation}
						\label{eq:elipt:rownowazn:gen:2}
						\begin{split}
							- \sum_{i=1}^N \int_{\Omega_i} \nabla \cdot \Big( \varepsilon_i \nabla u_i \Big) \varphi_i &= \int_\Omega f \varphi,
							\quad
							\forall \varphi \in L_2(\Omega)
							\\
							\Big[\varepsilon \nabla u \cdot \nu \Big] |_e & = 0 \quad \forall e \in \Gamma_I,
							\\
							\nabla u \cdot \nu & = 0 \quad \mbox{on } \partial \Omega_N.
						\end{split}
					\end{equation}
			\end{itemize}
		\end{lemma}

		\begin{proof}
			$\eqref{eq:elipt:rownowazn:gen:2} \Rightarrow \eqref{eq:elipt:rownowazn:gen:1}$ follows simply from the Green formula. 
			To prove $\eqref{eq:elipt:rownowazn:gen:1} \Rightarrow \eqref{eq:elipt:rownowazn:gen:2}$, take any $\varphi \in C_0^\infty(\Omega)$. Since $C_0^\infty(\Omega) \subset H_{0,\partial \Omega_D}^1(\Omega)$, then by \eqref{eq:elipt:rownowazn:gen:1} we have
			\begin{equation}
				\int_\Omega \varepsilon \nabla u \cdot \nabla \varphi = \int_\Omega f \varphi.
			\end{equation}
			By the Green formula
			\begin{equation}
				\begin{split}
					\int_\Omega f \varphi\, dx
					&=
					\sum_{i=1}^N
						\int_{\Omega_i} \varepsilon \nabla u \cdot \nabla \varphi\, dx
					\\ & =
					-\sum_{i=1}^N
						\int_{\Omega_i} \nabla \cdot \Big( \varepsilon \nabla u \Big) \varphi\, dx
					+\sum_{e \in \Gamma}
						\int_e [\varepsilon \nabla u \cdot \nu] \varphi\, ds
					.
				\end{split}
			\end{equation}
			Since $\varphi$ is zero on $\partial \Omega$, we may rewrite last sum as
			\begin{equation}
				\begin{split}
					\int_\Omega f \varphi\, dx
					+\sum_{i=1}^N
						\int_{\Omega_i} \nabla \cdot \Big( \varepsilon \nabla u \Big) \varphi\, dx
					& =
						\sum_{e \in \Gamma_I}
						\int_e [\varepsilon \nabla u \cdot \nu] \varphi\, ds
					.
				\end{split}
			\end{equation}
			Note that we may threat this relationship as an equality of distributions. Since $f, \nabla \cdot \big( \varepsilon \nabla u \big)\in L_2(\Omega) = \Big( L_2(\Omega) \Big)^*$, left-hand side clearly defines a linear continuous functional over $L_2(\Omega)$, while right-hand side does not unless it is identically zero, as for example it does not converge to zero provided that $\|\varphi\|_{L_2(\Omega)} \rightarrow 0$. Since sum of any two elements of a conjugated space cannot give an element not included in this space, both sides of the above equality must be zero. Thus in particular 
			\begin{equation}
				-\sum_{i=1}^N
					\int_{\Omega_i} \nabla \cdot \Big( \varepsilon \nabla u \Big) \varphi\, dx
				=
				\int_\Omega f \varphi
				.
			\end{equation}
			This statement is true for $\varphi \in C_0^\infty(\Omega)$. It is also true for any $\varphi \in L_2(\Omega)$ as $C_0^\infty(\Omega)$ is dense in $L_2(\Omega)$, and first statement of \eqref{eq:elipt:rownowazn:gen:2} is shown.
		\end{proof}

		\begin{proof}
			\emph{(Theorem \ref{thm:u:consistency:gen})}

			First we prove \eqref{eq:con:probl1:gen} $\Rightarrow$ \eqref{eq:con:probl2:gen}.
			Assume that $u^*$ is a solution of \eqref{eq:con:probl1:gen} and that it belongs to $H^1(\Omega) \cap H^2(\mathcal{E})$. We have by definition
			\begin{equation}
				\begin{split}
					\int_\Omega \varepsilon \nabla u^* \cdot \nabla \phi\, dx & = \int_\Omega f \phi\, dx
					\quad
					\forall \phi \in  H^1_{0,\partial \Omega_D}(\Omega).
				\end{split}
			\end{equation}
			We use lemma \ref{lem:elipt:rownowazn:gen} and we obtain that for any $\phi \in L_2(\Omega)$
			\begin{equation}
				- \int_\Omega \nabla \cdot \Big( \varepsilon \nabla u^*  \Big) \phi \, dx
				=
				\int_\Omega f \phi \, dx.
			\end{equation}

			Let us take any $\varphi \in H^1(\mathcal{E}) \cap H^2(\mathcal{T}_h)$ and substitute $\phi := \varphi$. We may split integrals to
			\begin{equation}
				- \sum_{i=1}^N \int_{\Omega_i} \nabla \cdot \Big( \varepsilon \nabla u^* \Big) \varphi\, dx
				=
				\sum_{i=1}^N \int_{\Omega_i} f \varphi\, dx
				.
			\end{equation}
			By the Green theorem, we have
			\begin{equation}
				- \int_{\Omega_i} \nabla \cdot \Big( \varepsilon \nabla u^* \Big) \varphi\, dx
				=
					\int_{\Omega_i} \varepsilon \nabla u^* \cdot \nabla \varphi\, dx
					-
					\int_{\partial \Omega_i} \varepsilon \nabla u^* \cdot \nu \varphi\, dx
				.
			\end{equation}
			Summing up these results in $\Omega_i$, we get
			\begin{equation}
				\sum_{i=1}^N 
					\int_{\Omega_i} \varepsilon \nabla u^* \cdot \nabla \varphi\, dx
				-
				\sum_{i=1}^N 
					\int_{\partial \Omega_i} \varepsilon \nabla u^* \cdot \nu \varphi\, dx
				=
					\sum_{i=1}^N \int_{\Omega_i} f \varphi\, dx
					.
			\end{equation}
			By lemma \ref{lem:elipt:rownowazn:gen}, we have that $[\varepsilon \nabla u^*] = 0$ on every $e \in \Gamma_I$, thus $\{\varepsilon \nabla u^* \cdot \nu\} = \varepsilon \nabla u^* \cdot \nu$ on any $\partial \Omega_i$ and we have
			\begin{equation}
				\sum_{i=1}^N 
					\int_{\Omega_i} \varepsilon \nabla u^* \cdot \nabla \varphi\, dx
				-
				\sum_{e \in \Gamma} 
					\int_{e} \Big\{ \varepsilon \nabla u^* \cdot \nu \Big\} [\varphi]\, dx
				=
					\sum_{i=1}^N \int_{\Omega_i} f \varphi\, dx
					.
			\end{equation}
			By the homogeneous Neumann boundary condition (lemma \ref{lem:elipt:rownowazn:gen}) on $e \in \Gamma_N$ we have $\{\varepsilon \nabla u^* \cdot \nu\} = 0$ and
			\begin{equation}
				\label{eq:kogoojiepheiyaighoophie:gen}
				\sum_{i=1}^N 
					\int_{\Omega_i} \varepsilon \nabla u^* \cdot \nabla \varphi\, dx
				-
				\sum_{e \in \Gamma_{DI}} 
					\int_{e} \Big\{ \varepsilon \nabla u^* \cdot \nu \Big\} [\varphi]\, dx
				=
					\sum_{i=1}^N \int_{\Omega_i} f \varphi\, dx
					.
			\end{equation}
			Since $u^* \in H^1(\Omega)$, then $[u^*]=0$ for any $e \in \Gamma_I$ and by assumption on $e \in \Gamma_D$ we have $u^* = \hat{u}$ so we have for any $\varphi \in H^1(\mathcal{E})$
			\begin{equation}
				\begin{split}
						\sum_{e \in \Gamma_{DI}}
							\eta_{e} \int_e [u^*] [\varphi]\, ds
						&
						-
						\sum_{e \in \Gamma_{DI}}
							\int_e \Big\{ \varepsilon \nabla \varphi \cdot \nu \Big\} [u^*]\, ds
					\\ &
					=
						\sum_{e \in \Gamma_{D}}
							\eta_{e} \int_e [\hat{u}] [\varphi]\, ds
						-
						\sum_{e \in \Gamma_{D}}
							\int_e \Big\{ \varepsilon \nabla \varphi \cdot \nu \Big\} [\hat{u}]\, ds
					.
				\end{split}
			\end{equation}
			By adding this result side-by-side to \eqref{eq:kogoojiepheiyaighoophie:gen} we obtain \eqref{eq:con:probl2:gen}.

			We proceed to \eqref{eq:con:probl2:gen} $\Rightarrow$ \eqref{eq:con:probl1:gen}.
			Assume \eqref{eq:con:probl2:gen} is true.
			First, we recover the Dirichlet boundary conditions. 
			Take any $e \in \Gamma_D$, such that $e \in \partial \Omega_i$, and $\bar{\varphi} \in C^\infty_0(e)$. Then let $\{\varphi_\epsilon\}_\epsilon$ be a sequence of functions, such that
			\begin{alignat*}{3}
				\varphi_\epsilon & \in C^\infty(\Omega),
				\quad
				&
				\varphi_\epsilon|_e &= \bar{\varphi},
				\quad
				&
				\supp(\varphi_\epsilon) & \subset \Omega_i \cup e,
				\quad
				\\
				\varphi_\epsilon|_{\partial \Omega_i \backslash e} & \equiv 0,
				\quad
				&
				\nabla \varphi_\epsilon \cdot \nu \Big|_{\partial \Omega_i} & = 0,
				\quad
				&
				\|\varphi_\epsilon\|_{L_2(\Omega)} & \xrightarrow[\epsilon \rightarrow 0]{} 0.
			\end{alignat*}
			Then $\varphi \in H^1(\mathcal{E}) \cap H^2(\mathcal{T}_h)$ and \eqref{eq:con:probl2:gen} becomes
			\begin{equation}
					\int_{\Omega_i} \varepsilon \nabla u^* \cdot \nabla \varphi_\epsilon\, dx
					-
					\int_e \varepsilon \nabla u^* \cdot \nu \bar{\varphi}\, ds
					+
					\eta_{e} \int_e u^*   \bar{\varphi}\, ds
				 =
					\int_{\Omega_i} f \varphi_\epsilon\, dx
					+
					\eta_{e} \int_e \hat{u} \bar{\varphi}\, ds
				.
			\end{equation}
			By the Green theorem
			\begin{equation}
					\int_{\Omega_i} \nabla \cdot \Big( \varepsilon \nabla u^* \Big) \varphi_\epsilon\, dx
					+
					\eta_{e} \int_e u^*   \bar{\varphi}\, ds
				 =
					\int_{\Omega_i} f \varphi_\epsilon\, dx
					+
					\eta_{e} \int_e \hat{u} \bar{\varphi}\, ds
				.
			\end{equation}
			Passing to the limit $\epsilon \rightarrow 0$
			\begin{equation}
					\eta_{e} \int_e u^* \bar{\varphi}\, ds
				 =
					\eta_{e} \int_e \hat{u} \bar{\varphi}\, ds
				.
			\end{equation}
			Since $\bar{\varphi} \in C_0^\infty(e)$ and $e \in \Gamma_D$ are arbitrary, we get
			\begin{equation}
				u^*|_{\partial \Omega_D} = \hat{u}|_{\partial \Omega_D},
			\end{equation}
			and the Dirichlet boundary conditions are satisfied.

			Then take any $\varphi \in C^\infty_{0,\partial \Omega_D} (\overline{\Omega})$. Thus 
			\begin{equation}
				\sum_{e \in \Gamma_{DI}} \eta_{e} \int_e[u^*] [\varphi]\, ds
				=
				\sum_{e \in \Gamma_D} \eta_{e} \int_e [\hat{u}] [\varphi]\, ds
				= 0.
			\end{equation}
			as $[\varphi] = 0$ for any $e \in \Gamma_I$ since $\varphi \in C^\infty_{0,\partial \Omega_D} (\overline{\Omega})$ and on $e \in \Gamma_D$ we have $[\varphi]=\varphi \equiv 0$.
			Analogously we see that
			\begin{equation}
				- \sum_{e \in \Gamma_{DI}}
					\int_e \Big\{\varepsilon \nabla u^* \cdot \nu\Big\} [\varphi]\, ds
					=
					0
				.
			\end{equation}
			By the assumptions of the theorem $u^* \in H^1(\Omega)$, so $[u^*] = 0$ for any $e \in \Gamma_I$ while as we have already been shown $u^* = \hat{u}$ for $e \in \Gamma_D$, so
			\begin{equation}
				- \sum_{e \in \Gamma_{DI}} \int_e \Big\{ \varepsilon \nabla \varphi \cdot \nu \Big\} [u^*]\, ds
				=
				- \sum_{e \in \Gamma_{D}} \int_e \Big\{ \varepsilon \nabla \varphi \cdot \nu \Big\} [\hat{u}]\, ds.
			\end{equation}
			Thus we obtain
			\begin{equation}
					\sum_{i=1}^N \int_{\Omega_i} \varepsilon \nabla u^* \cdot \nabla \varphi\, dx
				 =
					\int_\Omega f \varphi\, dx
				.
			\end{equation}
			Since this statement is true for any $\varphi \in C^\infty_{0,\partial \Omega_D} (\overline{\Omega})$, then it is valid also for any $\varphi \in H^1_{0,\partial \Omega_D}(\Omega)$, so we regain the first statement of \eqref{eq:con:probl1:gen}.
		\end{proof}

	\subsection{Auxiliary estimates}

		For better readability, we will divide the differential operators into few components. We define the following operators
		\begin{equation}
			\begin{split}
				A(u, \varphi)
				& :=
					\sum_{i=1}^N
						\int_{\Omega_i} \varepsilon \nabla u \cdot \nabla \varphi\, dx,
				\\
				B(u,\varphi)
				& :=
					\sum_{i=1}^N
						\int_{\Omega_i} \big( e^{u - \hat{v}} - e^{\hat{w} - u} \big)  \varphi\, dx,
				\\
				C(\varphi)
				& :=
					\sum_{i=1}^N
						\int_{\Omega_i} k_1 \varphi\, dx,
				\\
				D(u,\varphi)
				& :=
					- \sum_{e \in \Gamma_{DI}}
						\int_e \Big\{\varepsilon \frac{\partial u}{\partial \nu}\Big\} [\varphi]\, ds,
				\\
				E(u,\varphi)
				& :=
					- \sum_{e \in \Gamma_{DI}}
						\int_e \Big\{\varepsilon \frac{\partial \varphi}{\partial \nu}\Big\} [u]\, ds,
				\\
				F(\varphi)
				& :=
					- \sum_{e \in \Gamma_{D}}
						\int_e \{\varepsilon \nabla \varphi \cdot n \} [\hat{u}]\, ds,
				\\
				I(\varphi)
				& :=
					\sum_{e \in \Gamma_{D}}
						\eta_{e} \int_e [\hat{u}] \cdot [\varphi]\, ds,
				\\
				J(u,\varphi)
				& :=
					\sum_{e \in \Gamma_{DI}}
						\eta_{e} \int_e [u] \cdot [\varphi]\, ds.
			\end{split}
		\end{equation}
		In this section, we will prove several estimates for these operators. These estimates will be used in derivation of main result in Section \ref{sec:main:estimate}.

		\begin{lemma}
			\label{lemma:ADEJ:nonnegative}
			Let $u_h \in X_h(\Omega)$. Then
			\begin{equation}
				\begin{split}
					A & (u_h, u_h)
					+ J(u_h, u_h)
					+ D(u_h, u_h)
					+ E(u_h, u_h)
					\geq
					c \|u_h\|_h^2.
				\end{split}
			\end{equation}
		\end{lemma}

		\begin{proof}
			It is a simple consequence of lemma \ref{lem:broken:koerc:symm}.
		\end{proof}

		\begin{lemma}
			\label{lemma:B:nonnegative}
			Let $u,v \in L_2(\Omega)$. Then $B(u,u-v) - B(v,u-v) \geq 0$.
		\end{lemma}

		\begin{proof}
			Since the exponential function is monotone, we have
			\begin{equation}
				\begin{split}
					B(u,u-v) - B(v,u-v)
					& =
						\int_\Omega e^{-\hat{v}} \big( e^u - e^v \big)  (u-v)\, dx
					\\ & 
						+
						\int_\Omega e^{\hat{w}} \big( e^{-v} - e^{-u} \big)  (u-v)\, dx
					\geq
						0
						.
				\end{split}
			\end{equation}
		\end{proof}

		\begin{lemma}
			\label{lemma:AJ:broken:estimate}
			Let $u, \varphi \in H^1(\mathcal{E})$. Then
			\begin{equation}
				\Big|
					A(u, \varphi)
					+ J(u, \varphi)
				\Big|
				\leq
				C \|u\|_{h} \|\varphi\|_{h}.
			\end{equation}
		\end{lemma}

		\begin{proof}
			It is a simple consequence of the Schwarz inequality.
		\end{proof}

		\begin{lemma}
			\label{lemma:B:broken:estimate}
			Let $u, v, \varphi \in H^1(\mathcal{E})$ and $\alpha \leq u, v \leq \beta$ for some $\alpha, \beta \in \mathbb{R}$. Then
			\begin{equation}
				\begin{split}
					\Big|
						B(u, \varphi)
						- B(v, \varphi)
					\Big|
					\leq
						C \|u - v\|_{h}
							\|\varphi\|_{h}
						,
				\end{split}
			\end{equation}
			where $C$ is a constant dependent on
			$\alpha$, $\beta$,
			$\|\hat{v}\|_{L_\infty(\Omega)}$ and
			$\|\hat{w}\|_{L_\infty(\Omega)}$.
		\end{lemma}

		\begin{proof}
			Note that the exponential function is locally Lipschitz-continuous, so since $u, v$ are bounded
			\begin{equation}
				\|e^{u} - e^{v}\|_{L_2(\Omega)}
				\leq
					C \|u - v\|_{L_2(\Omega)}
				.
			\end{equation}
			The same is true for
			$e^{-v} - e^{-u}$.
			Thus using the Schwarz inequality and Poincare inequality for the broken norm (lemma \ref{lemma:poincare:broken:norm:2d})
			\begin{equation}
				\begin{split}
					\Big|
						B(u, \varphi)
						- B(v, \varphi)
					\Big|
					& =
						\Big|
							\int_\Omega e^{-\hat{v}} \big( e^{u} - e^{v} \big) \varphi \, dx
							+
							\int_\Omega e^{\hat{w}} \big( e^{-v} - e^{-u} \big) \varphi \, dx
						\Big|
					\\ &
					\leq
						C \|u - v\|_{L_2(\Omega)}
						\|\varphi\|_{L_2(\Omega)}
					\leq
						C \|u - v\|_{h}
						\|\varphi\|_{h}
						.
				\end{split}
			\end{equation}
		\end{proof}

		\begin{lemma}
			\label{lemma:D:interpolation:estimate}
			Let $u \in H^{k+1}(\mathcal{E})$, $u_I := \Pi_h u$ (see Section \ref{sec:interpolation:convergence:2d}) and $\varphi_h \in X_h(\Omega)$. Then
			\begin{equation}
				\begin{split}
					|D(u - u_I, \varphi_h)|
					&
						\leq
							C h^k
								\Big( \sum_{i=1}^N |u_i|_{H^{k+1}(\Omega_i)}^2 \Big)^{1/2}
								\|\varphi_h\|_{h}
						.
				\end{split}
			\end{equation}
			Constant $C$ depends on $\varepsilon_M$ and $\sigma_0$.
		\end{lemma}

		\begin{proof}
			We have
			\begin{equation}
				\begin{split}
					D(u - u_I, \varphi_h)
					& =
						- \sum_{e \in \Gamma_{DI}}
							\int_e \Big\{\varepsilon \nabla ( u - u_I ) \cdot \nu \Big\} [\varphi_h]\, ds
					.
				\end{split}
			\end{equation}
			Let us take any $e \in \Gamma_{I}$, $e \in \partial \Omega_j \cap \partial \Omega_l$. Then the Schwarz inequality yields that
			\begin{equation}
				\begin{split}
					\int_e \Big\{\varepsilon \nabla ( u - u_I ) \cdot \nu \Big\}
					&
						[\varphi_h]\, ds
					\leq
						\varepsilon_M \|\{\nabla ( u - u_I ) \cdot \nu \}\|_{L_2(e)} \| [\varphi_h] \|_{L_2(e)}
					.
				\end{split}
			\end{equation}
			Then by lemma \ref{lemma:estimate:inter:e} we get
			\begin{equation}
					\begin{split}
					\|\{\nabla ( u - u_I ) \cdot \nu \}\|_{L_2(e)}
					& \leq
						(
							h_j^{k-1/2} |u_j|_{H^{k+1}(\Omega_j)}
							+ h_l^{k-1/2} |u_l|_{H^{k+1}(\Omega_l)}
						)
					\\ & \leq
						(
							h_j
							+ h_l
						)^{k-1/2}
						(
							|u_j|_{H^{k+1}(\Omega_j)}
							+ |u_l|_{H^{k+1}(\Omega_l)}
						)
					.
				\end{split}
			\end{equation}
			Therefore
			\begin{equation}
				\begin{split}
					\eta_{e}^{-1} \|\{\nabla ( u - u_I ) \cdot \nu \}\|_{L_2(e)}^2
					& \leq
						C \sigma_e^{-1} h_j h_l (h_j + h_l)^{2k-2}
						( |u_j|_{H^{k+1}(\Omega_j)} + |u_l|_{H^{k+1}(\Omega_l)} )^2
					\\ & \leq
						C h^{2k}
						( |u_j|_{H^{k+1}(\Omega_j)} + |u_l|_{H^{k+1}(\Omega_l)} )^2
					.
				\end{split}
			\end{equation}
			If $e \in \Gamma_D$, $e \in \partial \Omega_i$, then analogously we have
			\begin{equation}
					\eta_{e}^{-1} \|\{\nabla ( u - u_I ) \cdot \nu \}\|_{L_2(e)}^2
					\leq
						C \sigma_e^{-1} h_i^{2k}
						|u_i|_{H^{k+1}(\Omega_i)}^2
					\leq
						C h^{2k} |u_i|_{H^{k+1}(\Omega_i)}^2
					.
			\end{equation}
			Therefore by Schwarz inequality and the inequalities derived above we conclude that
			\begin{equation}
				\begin{split}
					&\sum_{e \in \Gamma_{DI}} 
						 \int_e \Big\{\varepsilon \nabla ( u - u_I )\cdot \nu  \Big\}
							[\varphi_h]\, ds
					\\ & \leq
						C 
							\Big( \sum_{e \in \Gamma_{DI}} \eta_{e}^{-1} \|\{\nabla ( u - u_I ) \cdot \nu \}\|_{L_2(e)}^2 \Big)^{1/2}
							\Big( \sum_{e \in \Gamma_{DI}} \eta_{e} \| [\varphi_h] \|_{L_2(e)}^2 \Big)^{1/2}
					\\ & \leq
						C h^k  
							\Big( \sum_{i=1}^N |u_i|_{H^{k+1}(\Omega_i)}^2 \Big)^{1/2}
							\|\varphi_h\|_{h}
						.
				\end{split}
			\end{equation}
			Constant $C$ is independent of $h$. It depends on $\sigma_0$, $\varepsilon_M$ and on the number of elements of $\Gamma_{DI}$.
		\end{proof}

		\begin{lemma}
			\label{lemma:E:interpolation:estimate}
			Let $u \in H^2(\mathcal{E})$, $u_I := \Pi_h u$ (see Section \ref{sec:interpolation:convergence:2d}) and $\varphi_h \in X_h(\Omega)$. Then
			\begin{equation}
				\begin{split}
					|E(u - u_I, \varphi_h)|
					& \leq
						C \|\varphi_h\|_{h}
						\Bigg[ \sum_{i=1}^N
							\Big(
								h_i^{2 k} + \sum_{\Omega_l \in \nb(\Omega_i)} \frac{h_i^{2 k+1}}{h_l}
							\Big)
							|u_i|_{H^{k+1}(\Omega_i)}^2
						\Bigg]^{1/2}
					.
				\end{split}
			\end{equation}
		\end{lemma}

		\begin{proof}
			By the Schwarz inequality
			\begin{equation}
				\begin{split}
					|E(u - u_I, \varphi_h)|
					& \leq
					\sum_{e \in \Gamma_{DI}}
						\int_e \Big| \Big\{\varepsilon \nabla \varphi_h  \cdot \nu\Big\} \Big|  \Big| [u - u_I] \Big|\, ds
					\\ & \leq
						\varepsilon_M \sum_{e \in \Gamma_{DI}}
							\|\{\nabla \varphi_h  \cdot \nu\}\|_{L_2(e)}
							\| [u - u_I]\|_{L_2(e)}
					.
				\end{split}
			\end{equation}
			Splitting this sum up we get
			\begin{equation}
				\begin{split}
					\|\{\nabla \varphi_h  \cdot \nu\}\|_{L_2(e)}
					& \leq
						\Big\|\nabla \varphi_h  \cdot \nu \Big|_{\Omega_i} \Big\|_{L_2(e)}
						+ \Big\|\nabla \varphi_h  \cdot \nu \Big|_{\Omega_l} \Big\|_{L_2(e)}
					.
				\end{split}
			\end{equation}
			Then using lemma \ref{lemma:interface:trace:estimates} we have
			\begin{equation}
				\begin{split}
					\Big\|\nabla \varphi_h \cdot \nu \Big|_{\Omega_i} \Big\|_{L_2(e)}^2
					& \leq 
						C h_i^{-1} \|\nabla \varphi_h \|_{L_2(\Omega_i)}^2
					\leq
						C h_i^{-1} \|\varphi_h \|_{h}^2
						.
				\end{split}
			\end{equation}
			On the other hand we see that
			\begin{equation}
				\begin{split}
					\|[u - u_I]\|_{L_2(e)}
					&\leq
						\Big\|u - u_I \big|_{\Omega_i}\Big\|_{L_2(e)}
						+ \Big\|u - u_I \big|_{\Omega_l}\Big\|_{L_2(e)}
					.
				\end{split}
			\end{equation}
			By lemma \ref{lemma:estimate:inter:e} we have 
			$
				\|u - u_I \big|_{\Omega_i}\|_{L_2(e)}^2
				\leq
					C h_i^{2k+1} |u |_{H^{k+1}(\Omega_i)}^2
			$,
			so
			\begin{equation}
				\begin{split}
					|E(u - u_I, \varphi_h)|
					&
						\leq
						C \|\varphi_h\|_{h}
						\Bigg[ \sum_{i=1}^N
							\Big(
								h_i^{2k} + \sum_{\Omega_l \in \nb(\Omega_i)} \frac{h_i^{2k+1}}{h_l}
							\Big)
							|u_i|_{H^{k+1}(\Omega_i)}^2
						\Bigg]^{1/2}
					.
				\end{split}
			\end{equation}
		\end{proof}

	\subsection{Main estimate}
		\label{sec:main:estimate}

		The differential problem \eqref{eq:con:probl2:gen} satisfies:
		\begin{equation}
			\label{eq:probl:diff}
			\begin{split}
				A(u^*,\varphi)
				&
				+ B(u^*,\varphi)
				+ D(u^*,\varphi)
				+ E(u^*,\varphi)
				+ J(u^*,\varphi)
				\\ & =
					C(\varphi)
					+ F(\varphi)
					+ I(\varphi)
					\qquad
					\forall
					\varphi \in H^1(\mathcal{E}) \cap H^2(\mathcal{T}_h)
				.
			\end{split}
		\end{equation}
		On the other hand, the family of discrete problems depending on parameter $h$ is defined as
		\begin{equation}
			\label{eq:probl:discr}
			\begin{split}
				A(u_h^*,\varphi_h)
				+ B(u_h^*,\varphi_h)
				&
				+ D(u_h^*,\varphi_h)
				+ E(u_h^*,\varphi_h)
				+ J(u_h^*,\varphi_h)
				\\ & =
					C(\varphi_h)
					+ F(\varphi_h)
					+ I(\varphi_h)
					\qquad
					\forall
					\varphi_h \in X_h
				.
			\end{split}
		\end{equation}
		We subtract these equations from each other with $\varphi := \varphi_h$
		and we obtain
		\begin{equation}
			\begin{split}
				A(u^*-u_h^*,\varphi_h)
				&
				+ B(u^*,\varphi_h)
				- B(u_h^*,\varphi_h)
				+ D(u^*- u_h^*,\varphi_h)
				\\ &
				+ E(u^*-u_h^*,\varphi_h)
				+ J(u^*-u_h^*,\varphi_h)
				=
					0
				.
			\end{split}
		\end{equation}
		This is equivalent to
		$
			\lhs = \rhs
			,
		$
		where
		\begin{equation}
			\begin{split}
				\lhs & :=
					A(u^*_I-u_h^*,\varphi_h)
					+ B(u^*_I,\varphi_h)
					- B(u_h^*,\varphi_h)
					+ D(u^*_I - u_h^*,\varphi_h)
					\\ &
					+ E(u^*_I-u_h^*,\varphi_h)
					+ J(u^*_I-u_h^*,\varphi_h)
					,
			\end{split}
		\end{equation}
		and
		\begin{equation}
			\begin{split}
				\rhs & :=
					A(u^*_I-u^*,\varphi_h)
					+ B(u^*_I,\varphi_h)
					- B(u^*,\varphi_h)
					+ D(u^*_I - u^*,\varphi_h)
					\\ &
					+ E(u^*_I - u^*,\varphi_h)
					+ J(u^*_I - u^*,\varphi_h)
					.
			\end{split}
		\end{equation}

		Let us take $\varphi_h := u^*_I - u_h^*$.
		Then lemma \ref{lemma:ADEJ:nonnegative} and lemma \ref{lemma:B:nonnegative} imply $\lhs \geq c \|u^*_I - u_h^*\|_{h}^2$.
		Also we may estimate $\rhs$ with lemmas \ref{lemma:AJ:broken:estimate}, \ref{lemma:B:broken:estimate}, \ref{lemma:D:interpolation:estimate} and \ref{lemma:E:interpolation:estimate} 
		\begin{equation}
			\begin{split}
				\rhs
					\leq
						C
						\|u^*_I - u_h^*\|_{h}
						\Bigg(
							\|u^*_I - u^*\|_{h}
							+ \Bigg[ \sum_{i=1}^N
								\Big(
									h_i^{2k} + \sum_{\Omega_l \in \nb(\Omega_i)} \frac{h_i^{2k+1}}{h_l}
								\Big)
								|u^*_i|_{H^{k+1}(\Omega_i)}^2
							\Bigg]^{1/2}
						\Bigg)
					.
			\end{split}
		\end{equation}
		Thus estimating  $\lhs = \rhs$ from below and above and dividing by $\|u^*_I - u^*\|_{h} > 0$ we obtain
		\begin{equation}
			\begin{split}
				\|u^*_I - u_h^*\|_{h}
				\leq
					C \Bigg(
						\|u^*_I - u^*\|_{h}
						+ \Bigg[ \sum_{i=1}^N
							\Big(
								h_i^{2k} + \sum_{\Omega_l \in \nb(\Omega_i)} \frac{h_i^{2k+1}}{h_l}
							\Big)
							|u^*_i|_{H^{k+1}(\Omega_i)}^2
						\Bigg]^{1/2}
					\Bigg)
				.
			\end{split}
		\end{equation}
		Thus by the triangle inequality
		and interpolation error estimate \eqref{eq:inter:est:2d:csipg}
		we have
		\begin{equation}
			\begin{split}
				\|u^*-u_h^*\|_{h}
				\leq &
					\|u^*-u^*_I\|_{h}+\|u^*_I-u_h^*\|_{h}
				\\
				\leq &
					C \Bigg(
						\sum_{i=1}^N \Big( h_i^{2k} + \sum_{\Omega_l \in \nb(\Omega_i)} \frac{h_i^{2k+1}}{h_l}
						\Big) |u^*_i|_{H^{k+1}(\Omega_i)}^2
					\Bigg)^{1/2}
				.
			\end{split}
		\end{equation}
		Theorem \ref{thm:main:estimate} is therefore proven.
		For remark \ref{remark:main:estimate}, we assume that $h_i := c_i h$ for every $\Omega_i \in \mathcal{E}$
		and this estimate simplifies to
		\begin{equation}
			\begin{split}
				\|u^*-u_h^*\|_{h}
				\leq &
					C h^k \Big(
						\sum_{i=1}^N |u^*_i|_{H^{k+1}(\Omega_i)}^2 \Big)^{1/2}
				.
			\end{split}
		\end{equation}

\section{Numerical experiments}
	\label{sec:numerical:experiments}

	We would like to check whether the error estimate derived in Section \ref{sec:error:estimates} can be obtained in numerical simulations. Therefore we present two examples. These examples are not directly related to any specific semiconductor material. Simulations of the realistic semiconductor devices require accounting for material parameters and physical phenomena. These modifications do not substantially change the van Roosbroeck equations, but they go beyond the simplified model considered in this study.

	\begin{figure}
		\centering
		\includegraphics[width=0.7\textwidth]{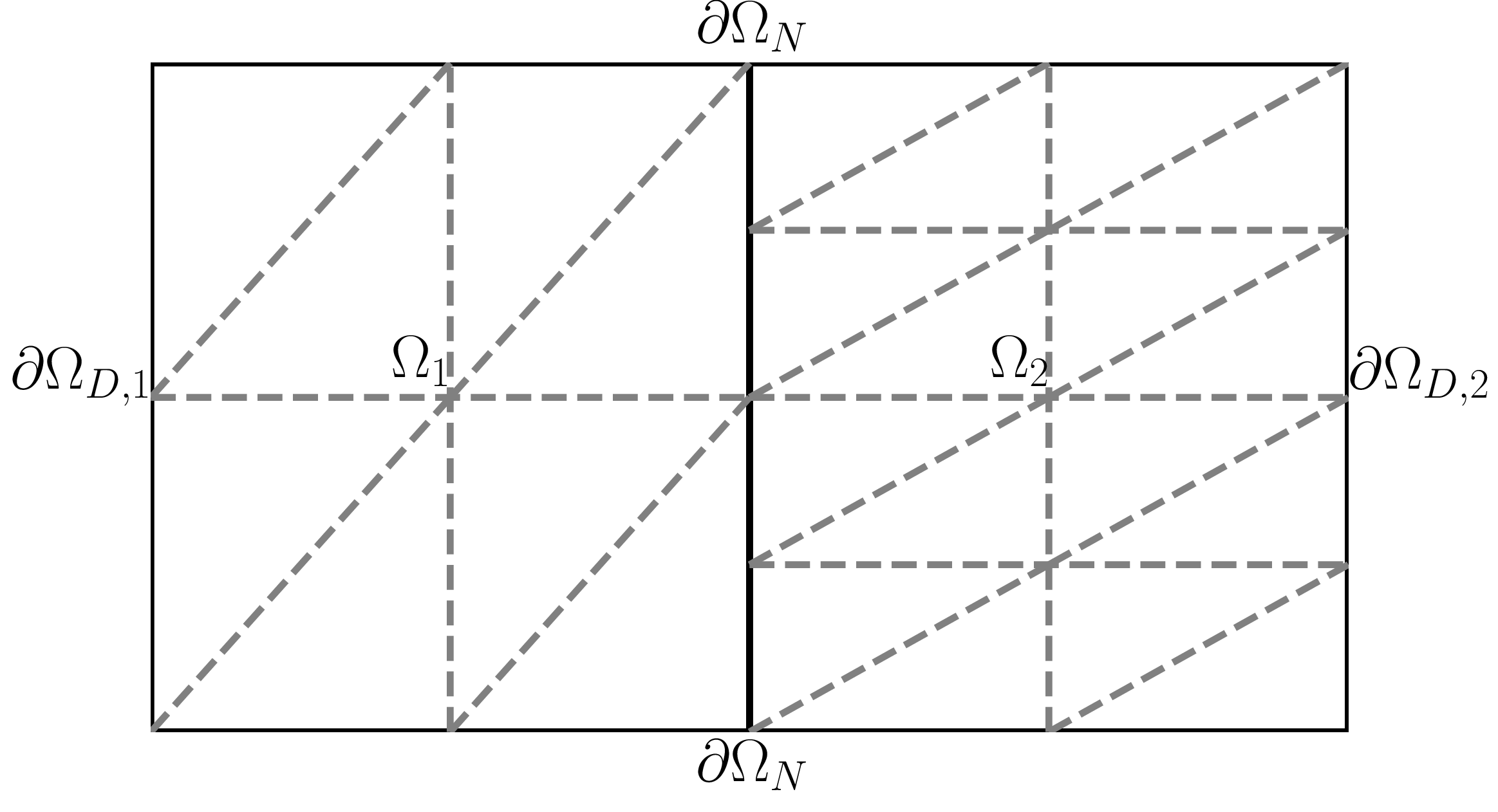}
		\caption{\label{rys:example009:K0001:grid} Schema of the first device used in simulations. It has two layers, corresponding to n-type layer $\Omega_1$ and p-type layer $\Omega_2$. Grid for $K=1$ is presented.}
	\end{figure}

	\begin{table}
		\centering
		\caption{\label{tab:example009} Parameters of the first device used in simulations. $N_x$ and $N_y$ denote number of nodes in horizontal and vertical direction, depending on parameter $K$. }
		\begin{tabular}{| l | r | r |}
			\hline
			Param.        & $\Omega_1$         & $\Omega_2$ \\
			\hline
			Length        & $\num{1e-2}$       & $\num{1e-2}$ \\ 
			Width         & $\num{1e-2}$       & $\num{1e-2}$ \\
			$N_x$         & $2 K + 1$          & $2 K + 1$  \\
			$N_y$         & $2 K + 1$          & $4 K + 1$  \\
			$\varepsilon$ & $\num{3e-3}$       & $\num{1e-3}$       \\
			$\mu_n$       & $\num{1e3}$        & $\num{3e3}$        \\
			$\mu_p$       & $\num{1e2}$        & $\num{3e2}$        \\
			$k_1$         & $\num{3e2}$        & $\num{-3e2}$       \\
			$C_{\rad}$    & $\num{1e-3}$       & $\num{2e-3}$       \\
			\hline
		\end{tabular}
	\end{table}

	Thus we will present simulations of abstract devices mimicking semiconductor p-n diodes.
	Our first example is a device, which consists of two layers $\Omega_1, \Omega_2$, corresponding to an n-type layer and a p-type layer of the p-n diode (Figure \ref{rys:example009:K0001:grid}). It has two contacts with metal electrodes, left and right, denoted by $\partial \Omega_{D,1}$ and $\partial \Omega_{D,2}$. Horizontal boundaries correspond to the contact with an insulator (e.g., air). Parameters of the device are presented in Table \ref{tab:example009}. We would thicken the grid with parameter $K$. For $K=1$, we divide both layers into two pieces in the horizontal direction, while in vertical direction $\Omega_1$ is divided into two pieces, while $\Omega_2$ is divided into four pieces (see Figure \ref{rys:example009:K0001:grid}). The grid nodes are distributed uniformly in horizontal and vertical direction within a given $\Omega_i$, and their number depends on the parameter $K$ as indicated by parameters $N_x, N_y$ in Table \ref{tab:example009}.

	In these simulations we assume that the operator $Q$ of equation \eqref{eq:dd} is some given piecewise-constant function:
	\begin{equation}
		Q(x,u,v,w) := C_{\rad}(x).
	\end{equation}
	This form corresponds to the radiative recombination \cite{Book-Selberherr-1984}. This physical process is responsible for emitting the light by a device.

	We start with the equilibrium state. Then the boundary conditions are as follows:
	$\hat{u}|_{\partial \Omega_{D,1}} = 0$
	and
	$\hat{u}|_{\partial \Omega_{D,2}} = u_{\built}$,
	where $u_{\built}$ is called a \emph{built-in potential}. It is chosen such that the charge defined as
	\begin{equation}
		\label{eq:rho}
		\rho(x) := k_1(x) - n(x) + p(x),
	\end{equation}
	is zero on $\partial \Omega_{D,2}$ if $u \equiv u_{\built}$.
	Here $n, p$ are the concentration of electrons and concentration of holes, defined as
	\begin{equation}
		\label{eq:def:n:p}
		n(x) := e^{u(x) - v(x)},
		\quad
		p(x) := e^{w(x) - u(x)}.
	\end{equation}
	This is a standard choice of the boundary conditions for the equilibrium state and it is motivated by physical arguments \cite{Book-Selberherr-1984}.
	Functions $v, w$ are constant, such that $\rho|_{\partial \Omega_{D,1}} = 0$.

	\begin{table}
		\caption{\label{tab:example009:error:step1} $L_2(\Omega)$- and $H^1(\Omega)$-error of $u$ in function of grid density parameter $K$ for the first device in equilibrium state. Numbers in brackets denote the error norm reduction factor.}
		\centering
\begin{tabular}{|r|r r|r r|}
	\hline
	K  & \multicolumn{2}{|c|}{$L_2(\Omega)$} & \multicolumn{2}{|c|}{$H^1(\Omega)$}\\
	\hline
	    1 & {      \num{4.6e-02}}  &                      & {      \num{3.0e-01}}  &              \\
	    2 & {      \num{1.1e-02}}  & \emph{          (4.0)} & {      \num{1.5e-01}}  & \emph{          (2.0)}\\
	    4 & {      \num{2.9e-03}}  & \emph{          (4.0)} & {      \num{7.6e-02}}  & \emph{          (2.0)}\\
	    8 & {      \num{6.9e-04}}  & \emph{          (4.1)} & {      \num{3.7e-02}}  & \emph{          (2.0)}\\
	   16 & {      \num{1.5e-04}}  & \emph{          (4.7)} & {      \num{1.7e-02}}  & \emph{          (2.2)}\\
	\hline
\end{tabular}

	\end{table}

	\begin{figure}
		\centering
		\includegraphics[width=0.95\textwidth]{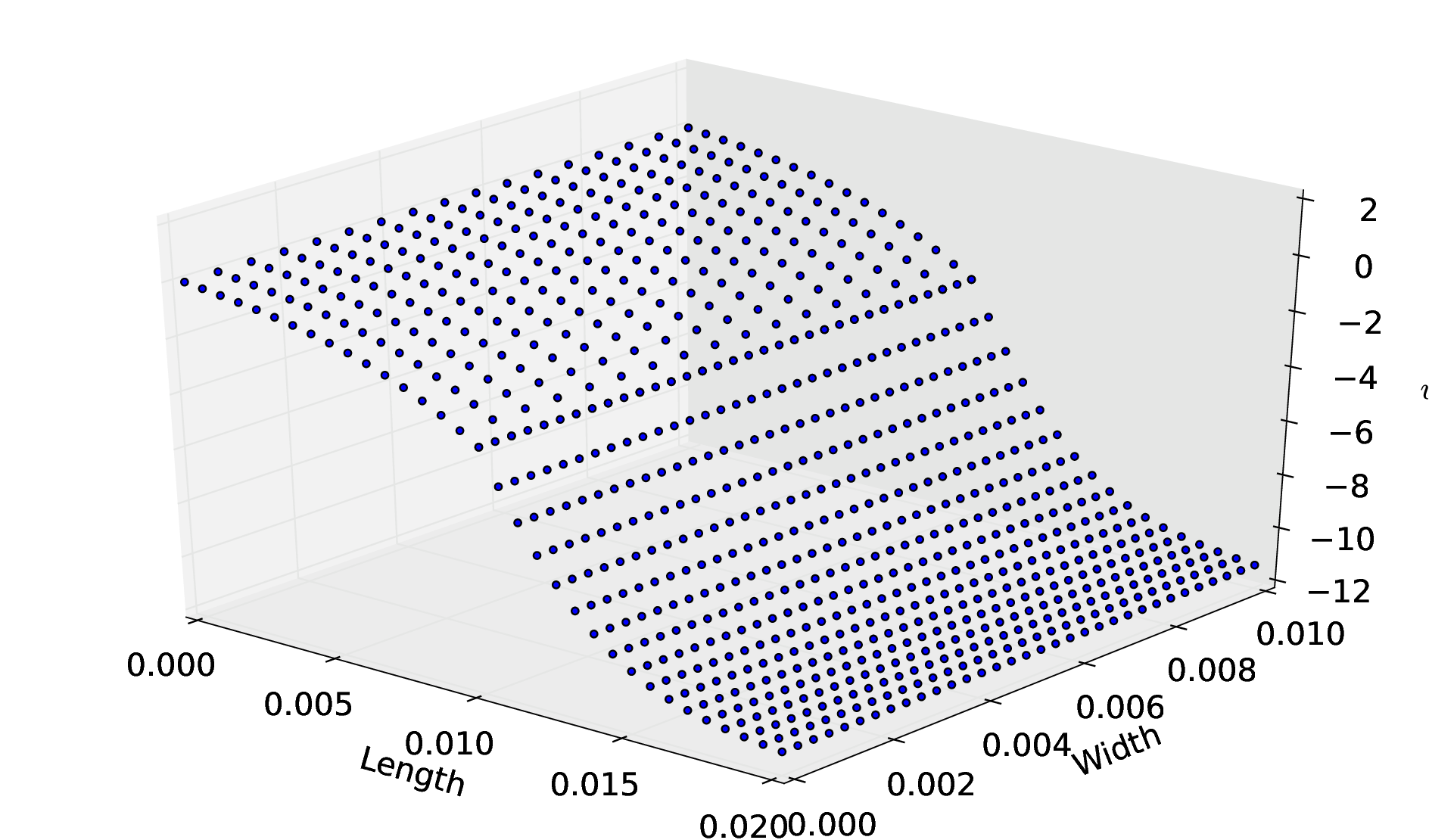}
		\caption{\label{rys:example009:K0008:step1:u} Function $u$ for the first example in the equilibrium state for $K=8$. Note one-dimensional character of the solution.}
	\end{figure}

	Simulations were performed for $K \in \{1, 2, 4, 8, 16, 32\}$, where $K=32$ is treated as a reference ``exact'' solution, i.e., 
	\begin{equation}
		\error_{K,L_2(\Omega)} := \|u_K - u_{32}\|_{L_2(\Omega)},
		\quad
		\error_{K,H^1(\Omega)} := \|u_K - u_{32}\|_{H^1(\Omega)},
	\end{equation}
	where $u_K := u_h$ for the grid parameter $K$.

	The nonlinear discrete problem was solved with the Newton method with step scaling relying on the Picard method. More details on nonlinear solver used in our simulations may be found in \cite{LectNotesComputSci-2014-8385-551}. 

	Results of these simulations are presented in Table \ref{tab:example009:error:step1}. We observe a linear reduction of the $H^1$-error, which is consistent with our theoretical result, as the $H^1$-norm is bounded by the broken norm up to a constant factor. We also note the quadratic $L_2$-norm convergence rate.
	These results were obtained for penalty parameter $\sigma_e = \num{3e6}$.

	\begin{table}
		\centering
		\caption{\label{tab:example012} Parameters of second device used in simulations. }
		\begin{tabular}{| l | r | r ||r | c | c|}
			\hline
			Param.         & $\Omega_1, \Omega_2, \Omega_3, \Omega_4, \Omega_7$       & $\Omega_5, \Omega_6, \Omega_8, \Omega_9$ & Grid & $N_x$ & $N_y$ \\  
			\hline                                                   
			Length         & $\num{1e-2}$       & $\num{1e-2}$  & $\Omega_1$ & $2K+1$ & $2K+1$ \\ 
			Width          & $\num{1e-2}$       & $\num{1e-2}$  & $\Omega_2$ & $2K+1$ & $2K+1$ \\
			$\varepsilon$  & $\num{3e-3}$       & $\num{1e-3}$  & $\Omega_3$ & $2K+1$ & $2K+1$ \\
			$\mu_n$        & $\num{1e3}$        & $\num{3e3}$   & $\Omega_4$ & $2K+1$ & $2K+1$ \\
			$\mu_p$        & $\num{1e2}$        & $\num{3e2}$   & $\Omega_5$ & $4K+1$ & $4K+1$ \\
			$k_1$          & $\num{3e2}$        & $\num{-3e2}$  & $\Omega_6$ & $2K+1$ & $4K+1$ \\
			$C_{\rad}$     & $\num{1e-3}$       & $\num{2e-3}$  & $\Omega_7$ & $2K+1$ & $2K+1$ \\
			               &                    &               & $\Omega_8$ & $4K+1$ & $2K+1$ \\
			               &                    &               & $\Omega_9$ & $2K+1$ & $2K+1$ \\
			\hline
		\end{tabular}
	\end{table}

	\begin{figure}
		\centering
		\includegraphics[width=0.9\textwidth]{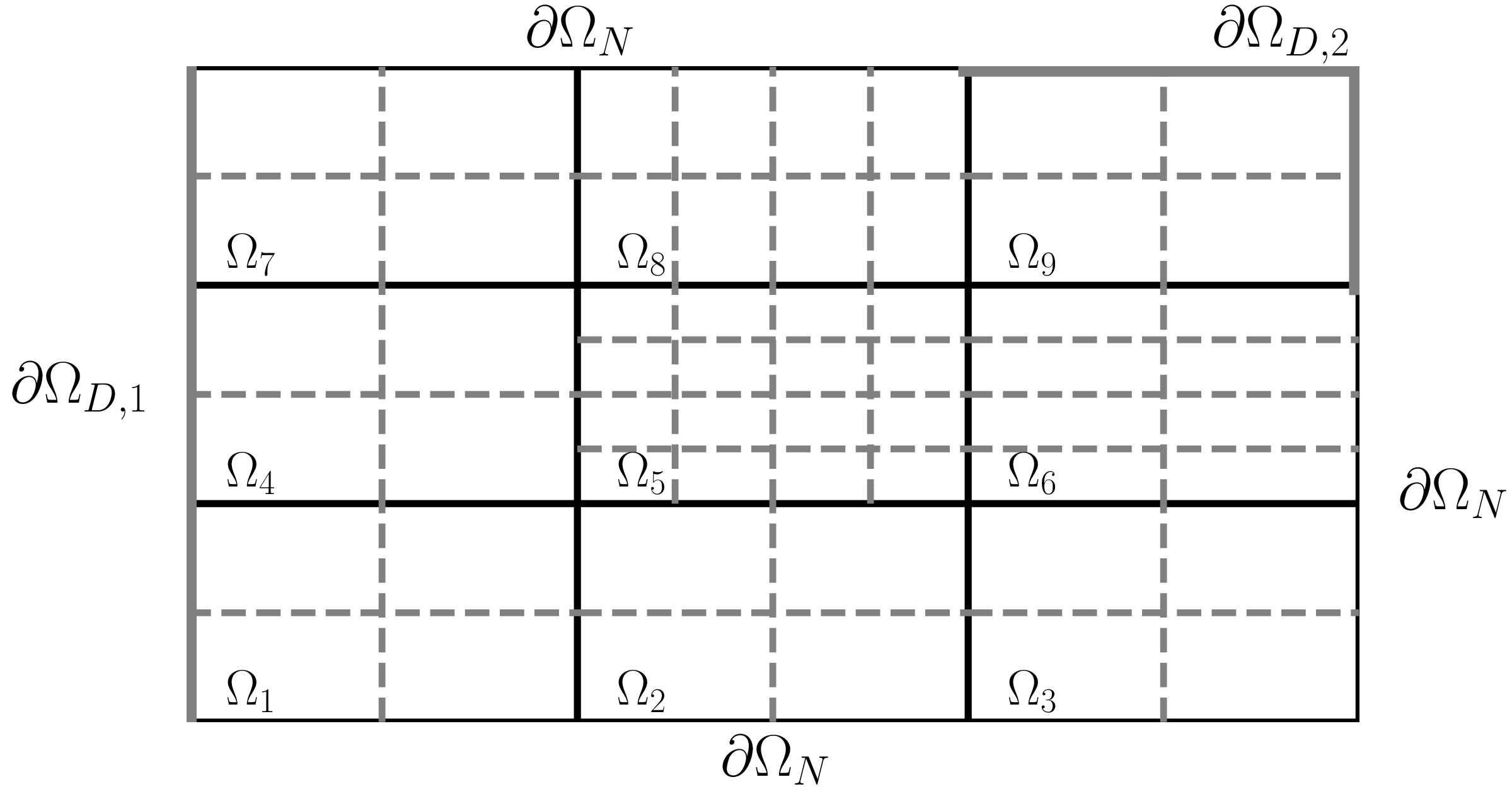}
		\caption{\label{rys:example012:K0001:grid} Schema of the second device used in simulations. Layers $\Omega_1, \Omega_2, \Omega_3, \Omega_4, \Omega_7$ correspond to the n-type region, while the remainder correspond to the p-type region. Left contact is attached to whole left edge, while right contact is attached to the boundary of $\Omega_9$. Grid for $K=1$ is presented with diagonal lines removed to improve readability.}
	\end{figure}

	\begin{table}
		\caption{\label{tab:example012:error:step1} $L_2(\Omega)$- and $H^1(\Omega)$-error of $u$ in function of grid density parameter $K$ for the second device in equilibrium state. Numbers in brackets denote the error norm reduction factor. }
		\centering
\begin{tabular}{|r|r r|r r|}
	\hline
 & \multicolumn{4}{|c|}{CSIPG}\\
	\hline
	K  & \multicolumn{2}{|c|}{$L_2(\Omega)$} & \multicolumn{2}{|c|}{$H^1(\Omega)$}\\
	\hline
	    1 & {      \num{2.0e-02}}  &                      & {      \num{1.9e-01}}  &              \\
	    2 & {      \num{5.2e-03}}  & \emph{          (3.8)} & {      \num{9.5e-02}}  & \emph{          (2.0)}\\
	    4 & {      \num{1.3e-03}}  & \emph{          (3.9)} & {      \num{4.8e-02}}  & \emph{          (2.0)}\\
	    8 & {      \num{3.5e-04}}  & \emph{          (3.8)} & {      \num{2.4e-02}}  & \emph{          (2.0)}\\
	   16 & {      \num{9.4e-05}}  & \emph{          (3.7)} & {      \num{1.2e-02}}  & \emph{          (2.0)}\\
	   32 & {      \num{2.3e-05}}  & \emph{          (4.0)} & {      \num{5.4e-03}}  & \emph{          (2.2)}\\
	\hline
\end{tabular}

	\end{table}

	As can be observed in Figure \ref{rys:example009:K0008:step1:u}, in this case, the solution has a one-dimensional nature.
	To study more sophisticated behavior, we introduce a second device with a more complex structure (Figure \ref{rys:example012:K0001:grid}, see Table \ref{tab:example012} for ``material'' parameters and grid description). 
	As we see in Table \ref{tab:example012:error:step1}, the convergence rate is similar as in the previous example.
	
	\begin{figure}
		\centering
		\includegraphics[width=0.95\textwidth]{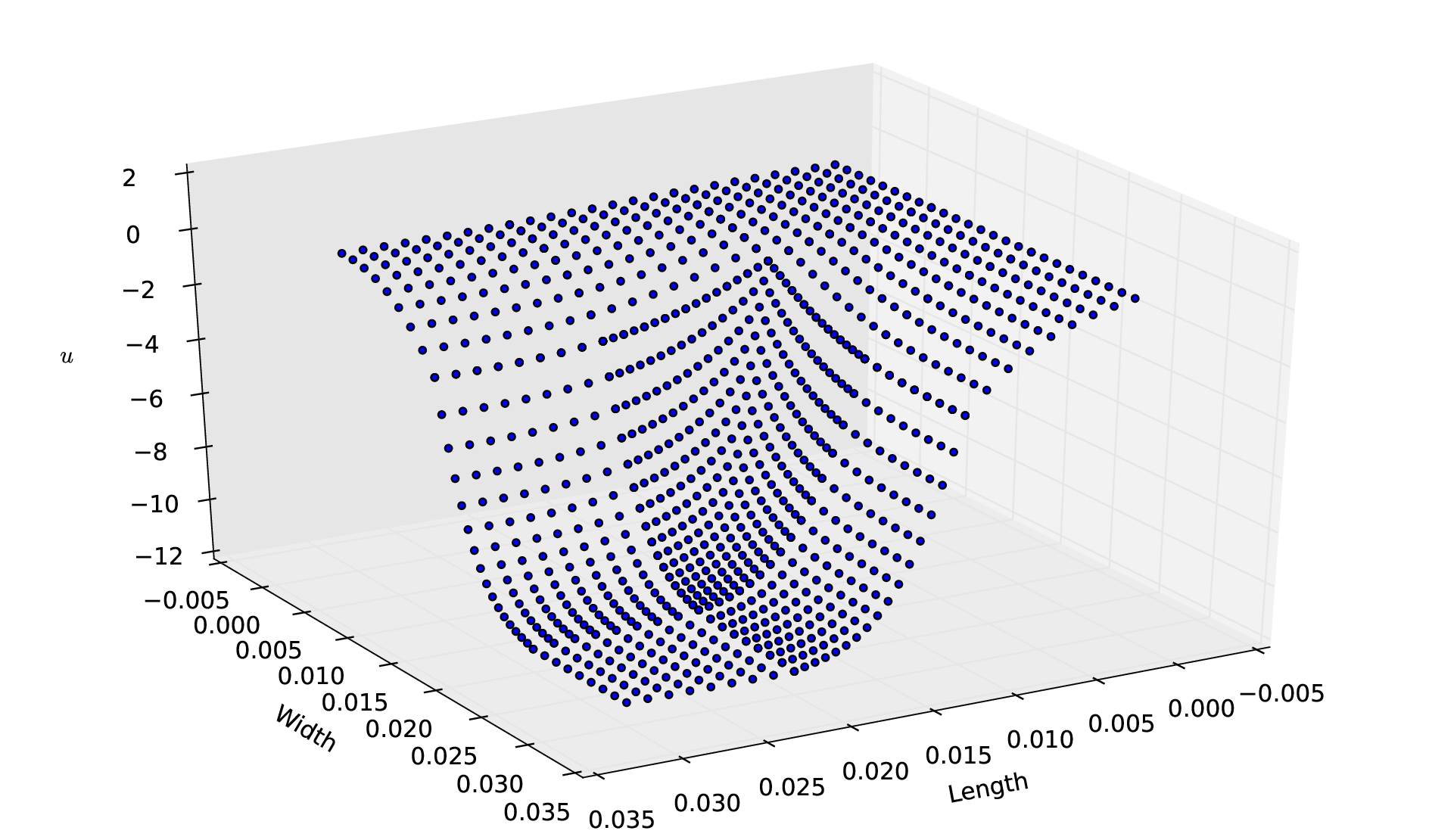}
		\caption{\label{rys:example012:K0004:step1:u} Function $u$ for the second example in the equilibrium state for $K=4$.}
	\end{figure}

	The theory presented in this paper covers only the equilibrium state, described in Section \ref{sec:differential:problem}. We also performed the simulations for the non-equilibrium state.
	Thus we use the presented discretization for every equation of system \eqref{eq:dd}.
	Boundary conditions on the function $u$ are similar as before, i.e.,
	$\hat{u}|_{\partial \Omega_{D,1}} = 0$
	and
	$\hat{u}|_{\partial \Omega_{D,2}} = u_{\built} + u_{\bias}$,
	where $u_{\bias}$ is a nonzero difference potential between the electrodes, called the bias.
	On functions $v, w$ we impose two implicit conditions on $\partial \Omega_D$:
	$v|_{\partial \Omega_D} = w|_{\partial \Omega_D}$
	and
	$\rho|_{\partial \Omega_D} = 0$, cf. \eqref{eq:rho}.
	On $\Omega_N$ we impose homogeneous Neumann boundary condition.

	\begin{table}
		\caption{\label{tab:example012:error:step4} $L_2(\Omega)$- and $H^1(\Omega)$-error of $u, v, w, n, p$ in function of grid density parameter $K$ for the second device for $u_{\bias} = 8$. Numbers in brackets denote the error norm reduction factor. Functions $n,p$ are defined in \eqref{eq:def:n:p}.}
		\centering
\begin{tabular}{|r|r r|r r|}
	\hline
 & \multicolumn{4}{|c|}{CSIPG}\\
	\hline
	K  & \multicolumn{2}{|c|}{$L_2(\Omega)$} & \multicolumn{2}{|c|}{$H^1(\Omega)$}\\
	\hline
	\multicolumn{5}{|c|}{Function: $u$} \\
	\hline
	    1 & {      \num{3.1e-02}}  &                      & {      \num{3.1e-01}}  &              \\
	    2 & {      \num{8.6e-03}}  & \emph{          (3.6)} & {      \num{1.6e-01}}  & \emph{          (1.9)}\\
	    4 & {      \num{2.6e-03}}  & \emph{          (3.3)} & {      \num{8.0e-02}}  & \emph{          (2.0)}\\
	    8 & {      \num{9.8e-04}}  & \emph{          (2.7)} & {      \num{4.0e-02}}  & \emph{          (2.0)}\\
	   16 & {      \num{4.0e-04}}  & \emph{          (2.5)} & {      \num{2.0e-02}}  & \emph{          (2.0)}\\
	   32 & {      \num{1.3e-04}}  & \emph{          (3.1)} & {      \num{8.8e-03}}  & \emph{          (2.2)}\\
	\hline
	\multicolumn{5}{|c|}{Function: $v$} \\
	\hline
	    1 & {      \num{1.7e-01}}  &                      & {      \num{9.8e-01}}  &              \\
	    2 & {      \num{9.7e-02}}  & \emph{          (1.7)} & {      \num{9.6e-01}}  & \emph{          (1.0)}\\
	    4 & {      \num{5.7e-02}}  & \emph{          (1.7)} & {      \num{9.1e-01}}  & \emph{          (1.0)}\\
	    8 & {      \num{3.2e-02}}  & \emph{          (1.8)} & {      \num{8.5e-01}}  & \emph{          (1.1)}\\
	   16 & {      \num{1.7e-02}}  & \emph{          (1.9)} & {      \num{7.5e-01}}  & \emph{          (1.1)}\\
	   32 & {      \num{6.9e-03}}  & \emph{          (2.4)} & {      \num{5.7e-01}}  & \emph{          (1.3)}\\
	\hline
	\multicolumn{5}{|c|}{Function: $w$} \\
	\hline
	    1 & {      \num{5.8e-01}}  &                      & {      \num{9.7e-01}}  &              \\
	    2 & {      \num{3.5e-01}}  & \emph{          (1.6)} & {      \num{9.5e-01}}  & \emph{          (1.0)}\\
	    4 & {      \num{2.1e-01}}  & \emph{          (1.6)} & {      \num{9.1e-01}}  & \emph{          (1.0)}\\
	    8 & {      \num{1.3e-01}}  & \emph{          (1.7)} & {      \num{8.5e-01}}  & \emph{          (1.1)}\\
	   16 & {      \num{6.8e-02}}  & \emph{          (1.9)} & {      \num{7.6e-01}}  & \emph{          (1.1)}\\
	   32 & {      \num{2.8e-02}}  & \emph{          (2.4)} & {      \num{5.9e-01}}  & \emph{          (1.3)}\\
	\hline
	\multicolumn{5}{|c|}{Function: $n$} \\
	\hline
	    1 & {      \num{4.7e-02}}  &                      & {      \num{4.5e-01}}  &              \\
	    2 & {      \num{1.5e-02}}  & \emph{          (3.1)} & {      \num{2.7e-01}}  & \emph{          (1.7)}\\
	    4 & {      \num{5.0e-03}}  & \emph{          (3.0)} & {      \num{1.5e-01}}  & \emph{          (1.8)}\\
	    8 & {      \num{1.7e-03}}  & \emph{          (2.9)} & {      \num{7.6e-02}}  & \emph{          (2.0)}\\
	   16 & {      \num{6.0e-04}}  & \emph{          (2.9)} & {      \num{3.8e-02}}  & \emph{          (2.0)}\\
	   32 & {      \num{1.7e-04}}  & \emph{          (3.5)} & {      \num{1.7e-02}}  & \emph{          (2.2)}\\
	\hline
	\multicolumn{5}{|c|}{Function: $p$} \\
	\hline
	    1 & {      \num{3.3e-02}}  &                      & {      \num{2.5e-01}}  &              \\
	    2 & {      \num{1.1e-02}}  & \emph{          (3.0)} & {      \num{1.2e-01}}  & \emph{          (2.1)}\\
	    4 & {      \num{4.6e-03}}  & \emph{          (2.4)} & {      \num{5.9e-02}}  & \emph{          (2.0)}\\
	    8 & {      \num{1.9e-03}}  & \emph{          (2.4)} & {      \num{3.0e-02}}  & \emph{          (2.0)}\\
	   16 & {      \num{7.1e-04}}  & \emph{          (2.7)} & {      \num{1.5e-02}}  & \emph{          (2.0)}\\
	   32 & {      \num{2.0e-04}}  & \emph{          (3.5)} & {      \num{6.5e-03}}  & \emph{          (2.3)}\\
	\hline
\end{tabular}

	\end{table}

	Results of this simulation are presented in Table \ref{tab:example012:error:step4}. For the function $u$, results are similar to the equilibrium state. For the functions $v, w$, the convergence is much worse. We may roughly estimate that the $L_2$-error reduces linearly, while the $H^1$-error convergence rate is sublinear, hard to estimate precisely without the exact solution.
	In the comparison, we also included the functions $n, p$. The van Roosbroeck equations may be formulated in terms of functions $u, v, w$, but from the physical point of view, there are other logical choices possible \cite{JNumerMethEng-1987-24-763}. Another choice is $u, n, p$ (see \eqref{eq:def:n:p} for the definition of $n, p$), as the charge $\rho$ and many recombination models (radiative, Shockley-Read-Hall, Auger) can be easily expressed in terms of these functions.

	We observe that the error convergence for $n,p$ is faster than for $v,w$, it is similar as for the function $u$.  Thus determination of physical parameters like the recombination rate, current or optical power, may rely on the better precision of functions $n,p$ despite the slow convergence of functions $v, w$.

\begin{tabular}{|r|r r|r r|r r|r r|}
	\hline
 & \multicolumn{4}{|c|}{P-N junction} & \multicolumn{4}{|c|}{Quantum well}\\
	\hline
	\rule{0pt}{10pt} 
	K  & \multicolumn{2}{|c|}{$L_2(\Omega)$} & \multicolumn{2}{|c|}{$H^1(\Omega)$} & \multicolumn{2}{|c|}{$L_2(\Omega)$} & \multicolumn{2}{|c|}{$H^1(\Omega)$}\\
	\hline
	    2 & {      1.8e-01}  &                        & {      6.9e-01}  &                        & {      7.5e-02}  &               & {      4.6e-01}  &              \\
	    4 & {      4.5e-02}  & \emph{          (4.0)} & {      4.2e-01}  & \emph{          (1.6)} & {      2.4e-02}  & \emph{          (3.1)} & {      2.6e-01}  & \emph{          (1.8)}\\
	    8 & {      2.3e-02}  & \emph{          (1.9)} & {      3.0e-01}  & \emph{          (1.4)} & {      7.3e-03}  & \emph{          (3.3)} & {      1.2e-01}  & \emph{          (2.2)}\\
	   16 & {      9.7e-03}  & \emph{          (2.4)} & {      2.1e-01}  & \emph{          (1.4)} & {      1.8e-03}  & \emph{          (4.0)} & {      5.9e-02}  & \emph{          (2.0)}\\
	   32 & {      2.8e-03}  & \emph{          (3.4)} & {      1.2e-01}  & \emph{          (1.8)} & {      4.5e-04}  & \emph{          (4.0)} & {      3.0e-02}  & \emph{          (2.0)}\\
	   64 & {      6.1e-04}  & \emph{          (4.6)} & {      5.7e-02}  & \emph{          (2.0)} & {      1.1e-04}  & \emph{          (4.0)} & {      1.5e-02}  & \emph{          (2.0)}\\
    128 & {      1.5e-04}  & \emph{          (4.2)} & {      2.9e-02}  & \emph{          (2.0)} & & & & \\
    256 & {      3.5e-05}  & \emph{          (4.2)} & {      1.4e-02}  & \emph{          (2.0)} & & & & \\
\hline
\end{tabular}




\section{Conclusions}
	\label{sec:conclusions}

	We have presented composite Discontinuous Galerkin discretization of the drift-diffusion equations, derived from Symmetric Interior Penalty Galerkin method \cite{Book-Riviere-2008}. The discrete problem is shown to be well-defined, and the error is estimated. In case of the uniform increase of grid density, the $H^1$-norm of error of Composite Symmetric Interior Penalty Galerkin (CSIPG) method error is estimated at $O(h)$. Results of numerical simulations presented in this paper agree with the theoretical estimates.

\section*{Acknowledgements}

	The authors acknowledge the support of the National Science Centre, Poland by Grant No. DEC-2016/21/B/ST1/00350.

\FloatBarrier

\end{document}